\documentclass[hidelinks,12pt]{amsart}
\usepackage{amstext, amsmath, amssymb, amsfonts, amsthm, mathtools, latexsym}
\usepackage{faktor}
\usepackage{graphicx}
\usepackage{hyperref}
\usepackage{enumitem}

\usepackage[left=3cm, right=3cm, bottom=4cm]{geometry} 
\setlist[itemize]{noitemsep, topsep=0pt}
\begin{document}

\newtheoremstyle{basico}
  {0,2cm}
  {0}
  {\itshape}
  {0,5cm}
  {\bfseries}
  {}
  {0,2cm}
  {\thmname{#1}\thmnumber{ #2}:\thmnote{ #3}}
\theoremstyle{basico}  

\newtheorem{teoprin}{Theorem}  
\newtheorem{coroprin}{Corollary}
\newtheorem{teo}{Theorem}[section]
\newtheorem{lema}[teo]{Lemma}
\newtheorem{prop}[teo]{Proposition}
\newtheorem{rem}[teo]{Remark}
\newtheorem*{conj}{Shub's conjecture}
\newtheorem{afirm}{Afirmation}

\newtheoremstyle{ejemplos}
  {0,2cm}
  {0}
  {}
  {0,5cm}
  {\bfseries}
  {}
  {0,2cm}
  {\thmname{#1}\thmnumber{ #2}:\thmnote{ #3}}
\theoremstyle{ejemplos}
\newtheorem{ex}{Example}

\newcommand{\B}{\mathbb{B}}
\newcommand{\N}{\mathbb{N}}
\newcommand{\Z}{\mathbb{Z}}
\newcommand{\Q}{\mathbb{Q}}
\newcommand{\D}{\mathbb{D}}
\newcommand{\R}{\mathbb{R}}
\newcommand{\C}{\mathbb{C}}
\newcommand{\E}{\mathbb{E}}
\newcommand{\T}{\mathbb{T}}
\newcommand{\A}{\mathbb{A}}
\newcommand{\pol}{\mathbb{P}}

\newcommand{\Or}{\mathcal{O}}
\newcommand{\OR}{\mathbb{O}}
\newcommand{\OOR}{\overline{\mathbb{O}}}

\newcommand{\Cy}{\mathcal{C}}
\newcommand{\Ha}{\mathcal{H}}

\title{ Orders of Growth and Generalized Entropy}

\author{Javier Correa}
\address{Universidade Federal de Minas Gerais}
\email {jcorrea@ufmg.mat.br}

\author{Enrique R. Pujals} 
\address{The Graduate Center - CUNY} 
\email {epujals@gc.cuny.edu}

\thanks{The first author has been supported by CAPES, and would like to thank UFRJ since this work started at his posdoctoral position in said university.}

\begin{abstract} We construct the complete set of orders of growth and we define on it the generalized entropy of a dynamical systems. With this object we provide a framework where we can study the separation of orbits of a map beyond the scope of exponential growth. We are going to show that this construction is particularly useful to study families of dynamical systems with vanishing entropy. Moreover, we are going to see that the space of orders of growth in which orbits are separated is wilder than expected. This is going to be achieved with different types of examples.
\end{abstract}

\maketitle

\section{Introduction.}


One of the goals in dynamical systems is to classify families of continuous maps according to their dynamical properties. One way to do this, is through a topological invariant, defined in an ordered set, which somehow measures the dynamical complexity of the systems. The notion of topological entropy, achieves this. It measures the exponential growth rate at which orbits of a system are separated. It was introduced by Adler, Konheim and McAndrew in \cite{AdKoAn65} and later on Dinaburg in \cite{Di71} and Bowen in \cite{Bo71} gave new equivalent definitions. 

The objective of this work is to provide a framework where we can generalize the classical notion of entropy, allowing a study beyond the scope of exponential growth. We are going to show that this construction is particularly useful to study families of dynamical systems with vanishing entropy. Moreover, we are going to see that the space of orders of growth in which orbits are separated is wilder than expected. And this, is going to be achieved with the different types of examples we are going to study. 

We shall begin this article with the construction of what we call the \emph{complete set of orders of growth}.

We start by considering the space of non-decreasing sequences in $[0,\infty)$, 
\[\Or=\{a:\N\to [0,\infty): a(n)\leq a(n+1)\ \forall n\in \N\}.\]
 In this space, we define an equivalence relation as follows: given $a_1,a_2\in \Or$ we say that $a_1\approx a_2$ if there exists $c,C\in(0,+\infty)$ such that $c a_1(n)\leq a_2(n)\leq C a_1(n)$ $\forall n\in \N$. The previous, is commonly written as $a_1\in \Theta (a_2)$ and the meaning for two sequences to be related, is that both of them have the same order of growth. Because of this, when we consider the quotient space $\OR = \faktor{\Or}{_\approx}$ we call it the space of orders of growth. If $a$ belongs to $\Or$ we are going to note $[a(n)]$ the class associated to $a$ which is an element of $\OR$. If we have a sequence defined by its formula  (for example $n^2$), we are going to represent the order of growth associated to it with the formula between brackets ($[n^2]\in \OR$). 

Since $\OR$ is the space of orders of growth, there is a clear notion of a order of growth being faster than another. This concept, defines a partial order in $\OR$ which we formalize through the following construction: given $[a_1(n)],[a_2(n)] \in \OR$ we say that $[a_1(n)] \leq [a_2(n)]$ if there exists $C>0$ such that $a_1(n) \leq C a_2(n)$. This partial order is well defined because it does not depend on the choices of $a_1$ and $a_2$. 

We have now $(\OR,\leq)$  which is a partial order. We recall that the properties that define a partial order are the reflexivity ($o \leq o, \forall o\in \OR$), antisymmetry (if $o_1\leq o_2$ and $o_2\leq o_1$, then $o_1 = o_2$) and the transitivity (if $o_1 \leq o_2$ and $o_2\leq o_3$, then $o_1 \leq o_3$). For our purposes, we would like to be able to take ``limits" in this space and therefore, we need to complete it. We say that a set $L$ with a partial order is a complete lattice if every subset $A\subset L$ has both an infimum and a supremum. We consider now $\OOR$ the Dedekind - MacNeille completion of $\OR$. This is the smallest complete lattice which contains $\OR$. In particular, it is uniquely defined and from now on, we will consider that $\OR\subset \OOR$. We will also call $\OOR$ the complete set of orders of growth. Another way to define $\OOR$ is to consider in $\OR$ the order topology and then consider the compactification of $\OR$ respecting the partial order.

Since $\OOR$ is not a complete order, just a partial order, we are not going to represent the elements of $\OOR$ in a line. We are going to represent them in the plane. Given $o,u\in \OOR$, if we design $o$ to the right of $u$, then $o$ and $u$ may or may not be comparable but if they are, $u<o$. However, if we design them on the same horizontal line and $o$ is to the right of $u$, then $u<o$ holds.


We want now to define entropy of a dynamical systems in the complete space of orders of growth. We will consider that the reader is familiar with the notion of topological entropy. Check for example \cite{ViOl16}, \cite{Wa82} or \cite{Ma87} for more details. Let us briefly recall the concepts involved. 

Given $M$ a compact metric space and $f:M\to M$ a continuous map we define the dynamical ball $B(x,n,\epsilon)=\{y\in M: d_n(x,y)\leq \epsilon\}$ where $d_n(x,y)=\sup\{d(f^i(x),f^i(y)):0\leq i \leq n\}$. A set $E\subset M$ is  $(n,\epsilon)$-generator if 
$M = \bigcup_{x\in E} B(x,n,\epsilon)$. By compactness of $M$ there always exists a finite $(n,\epsilon)$-generator set. We define then $g(f,\epsilon,n)$ as the smallest possible cardinality of a finite $(n,\epsilon)$-generator. If we fix $\epsilon>0$ , then we observe that $g(f,\epsilon,n)$ is an increasing sequence of natural numbers. And for a fixed $n$ if $\epsilon_1>\epsilon_2$, then $g(f,\epsilon_1,n)\geq g(f,\epsilon_2,n)$. 

We will set our notation as follows, the sequence $g_{f,\epsilon}\in \Or$ is defined by $g_{f,\epsilon}(n) =g(f,\epsilon,n)$. By the previous, we deduce that $[g_{f,\epsilon_1}(n)]> [g_{f,\epsilon_2}(n)]$ if $\epsilon_1<\epsilon_2$. If we consider $G_f=\{[g_{f,\epsilon}(n)]\in \OR:\epsilon > 0\}$, then we define the generalized topological entropy of $f$ as 
\[o(f)= ``\lim_{\epsilon \to 0}"[g_{f,\epsilon}(n)] =\sup(G_f)\in \OOR.\]

The first thing we want to state about generalized entropy is that its a topological invariant. 

\begin{teoprin}\label{Inv}
Let $M$ and $N$ be two compact metric spaces and $f:M\to M$, $g:N\to N$ two continuous map. Suppose there exists $h:M\to N$ a homeomorphism such that $h\circ f = g \circ h$. Then,  $o(f)=o(g)$.
\end{teoprin}


Recalling that the topological entropy of a map is defined as
\[h(f) = \lim_{\epsilon\to 0}\limsup_n \frac{1}{n} log(g_{f,\epsilon}(n)),\]
the natural question now is how the generalized topological entropy  is related to topological entropy. The answer to this question is very simple, the classical notion of topological entropy is the projection of the generalized entropy into the family of exponential orders of growth. 

The exponential orders of growth are the classes of the sequences $\{exp(tn)\}_{n\in\N}$ where $t$ is a number between $0$ and $\infty$. Then, the family of exponential orders of growth is the set 
\[ \E=\{[exp(tn)]: t\in (0,\infty)\}\subset \OR.\]
 Although, it is not necessary for now, we take the opportunity to remark that the elements $\inf(\E)$ and  $\sup(\E)$ belong to $\OOR$ and are both abstract orders of growth which are not realizable by any sequence. 

Once we have established the family of exponential growths $\E$, we say how we compare an element $o\in \OOR$ with $\E$. Given $o\in \OOR$ we consider the interval $I_{\E}(o)= \{t\in (0,\infty):o \leq [exp(tn)]\}\subset \R$. We would like to observe that the order of growth $o$ might not be comparable to any element of $\E$ and therefore the set $I_{\E}(o)$ might be the empty set. In any case, we define the projection  $\pi_{\E}:\OOR \to [0,\infty]$ by the following rule: 
\begin{itemize}
\item If $I_{\E}(o)\neq \emptyset$, then $\pi_{\E}(o)=\inf (I_{\E}(o))$.
\item If $I_{\E}(o)=\emptyset$, then $\pi_{\E}(o)=\infty$.
\end{itemize} 

Now that we defined how to project a order of growth into the family of exponential orders of growth, let us enunciate our second theorem.

\begin{teoprin}\label{Rel}
Let $M$ be a compact metric space and $f:M\to M$ a  continuous map. Then,  $\pi_{\E}(o(f))=h(f)$. And, $o(f)\leq sup(\E)$. 
\end{teoprin}

\begin{figure}[!htp]
\centering
\includegraphics{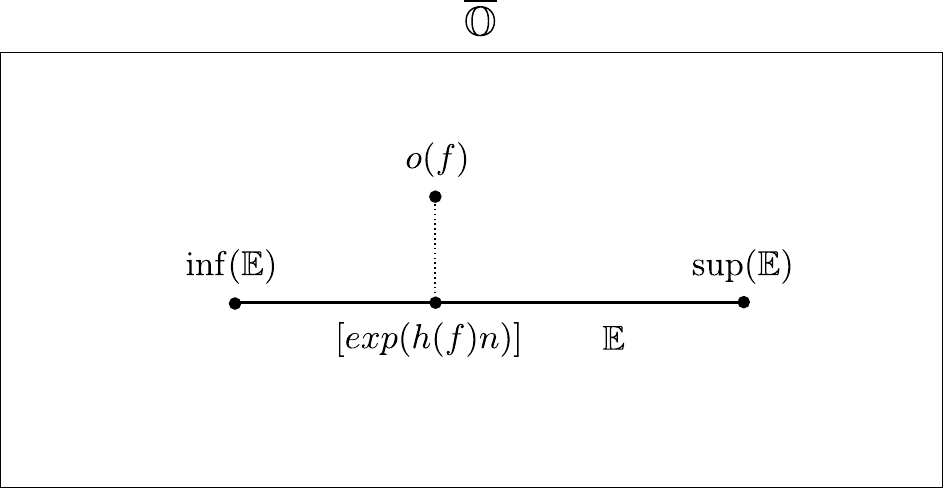}
\caption{Theorem \ref{Rel}}
\end{figure}

We would like to point out that we are projecting into the closure of the set of indexes that define $\E$ and not into $\E$ itself. The reason for this is that $\OOR$ is so big that $\E$ is not a closed set and it is in fact discrete.

Let us show some examples: 

\begin{ex} 
If $\Sigma_k = \{1,\dots,k\}^\N$ and $\sigma:\Sigma_k \to \Sigma_k$ is the shift, then we know that 
\[ g_{f,\epsilon}(n) = k^{(n+\left\lfloor 1/\epsilon\right\rfloor)} = exp(log(k)(n + \left\lfloor 1/\epsilon \right\rfloor)) = C(\epsilon). exp(log(k)n),\]
where $C(\epsilon)$ is a constant which depends only on epsilon. When we consider the order of growth associated to such sequence, we can ignore $C(\epsilon)$ and then $[g_{f,\epsilon}(n)] = [exp(log(k)n)]$ for all $\epsilon$. This implies that $o(\sigma)= [exp(log(k)n)]$. 
\end{ex}

The next example, shows a dynamical system such that its generalized entropy is an abstract order of growth (an element of $\OOR \setminus \OR$).

\begin{ex} 
Consider $\Sigma = [0,1]^\N$ and $\sigma:\Sigma\to \Sigma$ is the shift. In this case, it is not hard to see that 
\[g_{\sigma,\epsilon}(n) = (2/\epsilon)^{(n+\left\lfloor 1/\epsilon\right\rfloor)} = exp(log(2/\epsilon) (n + \left\lfloor 1/\epsilon\right\rfloor)) = C(\epsilon). exp(log(2/\epsilon)n). \]
 And from this, we deduce that $[g_{\sigma,\epsilon}(n)] =[exp(log(2/\epsilon)n)]$ and since $\{log(2/\epsilon): 0 < \epsilon < 1\} = (log(2),\infty)$, we conclude that $o(\sigma) = \sup(\E)$.
\end{ex}

We would like to recall that it is also possible to construct examples with $o(f)=\sup(\E)$ in the context of manifolds, with $C^0$ maps. 

It is interesting for us to know if there are other examples such that its generalized topological entropy is an abstract order of growth. We know that expansivity in the compact case is an obstruction for this phenomena (this is proved in appendix B). 


Since the space of maps such that $0<h(f)<\infty$ is relatively well understood, we ask what can we say in our context with maps such that $h(f)=\infty$ or $h(f)=0$. The first category has been answered in Theorem \ref{Rel}. The inequality $o(f) \leq \sup(\E)$, implies that $h(f)=\infty$ if and only if $o(f)=\sup(\E)$. In particular, from the standard perspective of separation of orbits, maps with infinite entropy can not be told apart.

On the other hand, much can be said when $h(f)=0$. For now, we are going to restrict ourselves to understand those systems that have the most simple dynamics. Let us introduce an important element of $\OOR$.  Since $\OOR$ is a complete lattice, it has a minimum. Nonetheless, the minimum of $\OOR$ already belongs to $\OR$ and it is the equivalence class of the constant sequence. To simplify the notation we are going to denote such element by $0$.  

It interest us to know, which are the maps such that $o(f)=0$. Our following theorem answers this question and shows a simple condition to obtain at least linear growth.

\begin{teoprin}\label{Ent0}
Let $M$ be a compact metric space and $f:M\to M$ a continuous map. Then, $o(f)=0$ if and only if $f$ is Lyapunov stable. In addition, if $f$ is a homeomorphism, and there exists $x\in M$ such that $x\notin \alpha(x)$, then $o(f) \geq [n]$. 
\end{teoprin}

The first part of Theorem \ref{Ent0}, has already been proved by Blanchard, Host and Massin in \cite{BlHoMa00} where the property $o(f)=0$ is called bounded complexity and Lyapunov stable maps are called equicontinuous. However, in this article we are also going to offer an alternate proof.   

From the second part, we conclude the following corollary

\begin{coroprin}
Let $f:M\to M$ be a continuous map on a metric compact space. If $o(f)<[n]$, then every point is recurrent and therefore $Rec(f)=\Omega(f)= M$. In particular, when $M$ is connected, $f$ has a point $x$ whose $\omega$ limit is not a periodic orbit. 
\end{coroprin}


Our next objective is to discuss how to classify dynamical systems through generalized topological entropy. At first glance, one would be tempted to say that $f$ is more chaotic than $g$ if $o(f)>o(g)$. This notion has two problems. Since there is no information loss when considering the generalized topological entropy, $o(f)$ can detect separation of orbits in places where topological entropy can not. For example, if $\Omega(f)$ is the no-wandering set of $f$, a simple conclusion from the variational principle is that  $h(f_{|\Omega(f)}) = h(f)$. However, in the context of generalized topological entropy, this is false. We naturally have that $o(f_{|\Omega(f)}) \leq o(f)$, yet there are examples where the inequality is strict. This means that $o(f)$ can detect separation of orbits in places like the wandering set, and so we consider that this should be taken into account. The strict inequality also holds between other important dynamical sets.

\begin{ex}\label{ExBE}
There exists a map $f:\D^2\to \D^2$ such that $o(f_{|\overline{Rec(f)}}) = 0,$ $o(f_{|\Omega(f)}) = [n]$ and $o(f)\geq [n^2]$. 
\end{ex} 

This example is constructed and explained in subsection \ref{SExBE} and therefore we move on with our discussion. The second problem we have, is that in the context of to\-po\-lo\-gi\-cal entropy, the word ``chaotic" is reserved for maps with positive entropy. However, in our context, we work mostly with maps with vanishing entropy and therefore we would prefer another word for maps with positive generalized entropy. Since generalized entropy implies separation orbits, we choose the word dispersion. Because of this, we propose the following criteria. We say that $f$ is more \emph{dispersive} than $g$ if 
\begin{itemize}
\item $ o(f_{|\Omega(f)})>  o(g_{|\Omega(g)})$ or
\item $ o(f_{|\Omega(f)})= o(g_{|\Omega(g)})$ and $o(f)>o(g)$.
\end{itemize}

We would like stress that we choose to focus in the no-wandering set and the whole space because of preference. It could be very well added in the discussion the limit set, the closure of the recurrent set, the chain recurrent set or the closure of the union of the supports of all the invariant measures. The choice of which sets are to be consider, should depend on the family of maps one is working with.  

We will call the tuple $(o(f_{|\Omega(f)}),o(f))$ the entropy numbers of $f$. With this criteria, we can prove the following:

\begin{teoprin}\label{ClaHS1}
In the space of homeomorphisms of the circle, there are three categories:
\begin{itemize}
\item either $f$  has entropy numbers $(0,0)$ and it is Lyapunov stable; 
\item or $f$ has entropy numbers $(0,[n])$, it is not Lyapunov stable and has periodic points;
\item or $f$ has entropy numbers $([n],[n])$ and it is a Denjoy map.
\end{itemize}
In particular, in the space of homeomorphisms of the circle, Denjoy maps are more dispersive than Morse-Smale maps which are more dispersive than rotations.
\end{teoprin}

We would like to recall that every homeomorphism of the circle has zero topological entropy. Therefore, with generalized topological entropy we can tell apart maps which are indistinguishable by topological entropy. 

With this perspective, we can not say that irrational rotations are dynamically more complex than rational rotations since both of them have entropy numbers  $(0,0)$. On the other hand, Morse-Smale maps have bigger entropy numbers than irrational rotations. Now, the extra complexity of the irrational rotations comes from the structure of the orbits and not from the separation of the orbits itslef. This implies that in the context of 
vanishing entropy, orbit structure and dispersion of orbits are not intrinsically related as in the context of positive entropy.  In particular, we take  that in the context of homeomorphisms of the circle, both the rotation number and the generalized entropy are the key to classify them.


We continue with our study of maps with vanishing entropy, through reviewing previous works. In all of them, it is studied the polynomial entropy of dynamical systems. From our point of view, polynomial entropy is not a sufficient tool to measure dispersion of orbits on maps with vanishing entropy and this is going to be shown in Theorem \ref{CyCa}. For now, we move into explaining what is polynomial entropy. This concept was introduced by Marco in the context of integrable Hamiltonian maps in \cite{Ma13} and the definition is
\[h_{pol}(f) =\lim_{\epsilon\to 0}\limsup_n \frac{log(g_{f,\epsilon}(n))}{log(n)}.\]

If we define the family of polynomial orders of growth by $\pol=\{[n^t]\in \OR: t\in (0,\infty)\}$, then by the arguments of Theorem \ref{Rel} we infer that
\[\pi_{\pol}(o(f))=h_{pol}(f).\]

Figure \ref{FigEspacio} is a representation of the set $\{o(f)\in \OOR:f\text{ is a continuous map}\}$ that we add to give some perspective.

\begin{figure}[!htp] 
\centering
\includegraphics[width=\textwidth]{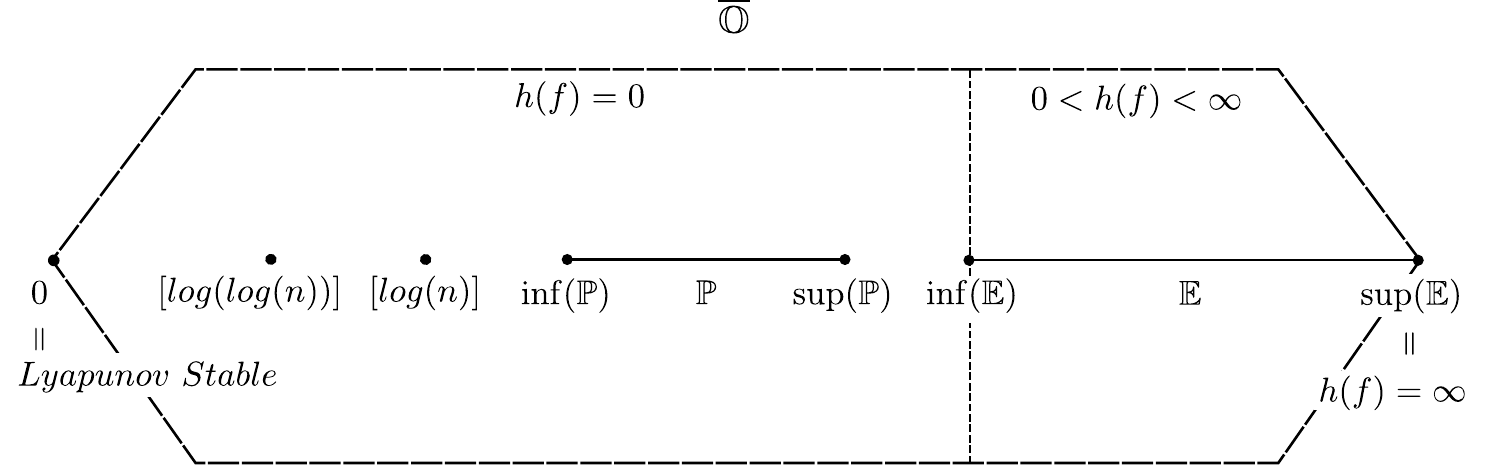}
\caption{$\{o(f)\in \OOR:f\text{ is a continuous map}\}$}
\label{FigEspacio}
\end{figure}

The polynomial entropy of a map has been studied first by  Labrousse in \cite{La13}. In her work, she studies the polynomial entropy of flows in the torus and the polynomial entropy of circle homeomorphisms. In particular, for the circle homeomorphisms she shows that the polynomial entropy is always $0$ or $1$ and that $0$ is only taken by homeomorphisms conjugate to a rotation. Theorem \ref{ClaHS1} is more general for two reasons. First, we take into account the no-wandering set. Secondly, we observe that saying that $o(f)=[n]$ is stronger than saying $h_{pol}(f)=1$ because, for example, $\pi_{\pol}([log(n)n])=1$. 

A second work in polynomial entropy is \cite{BeLa16} from Bernard and Labrousse where they study the polynomial entropy of geodesic 
flows for riemannian metrics on the two torus. There, they prove that the geodesic flow has polynomial entropy $1$ if and only if the torus is isometric to a flat torus.

After this, comes the work from Artigue, Carrasco-Olivera and  Monteverde \cite{ArCOMo18}, where they show two examples:
\begin{enumerate}
\item A continuous map $f:M\to M$, where $M$ is a compact metric space, such that $h_{pol}(f)=0$, yet $f$ is not Lyapunov stable.
\item For each $c>1$, a continuous map $f:M\to M$, where $M$ is a compact metric space, such that $\frac{1}{c+1} \leq h_{pol}(f)\leq\frac{1}{c}$.
\end{enumerate}

In our context, through their technique more can be said. In fact,  the first example satisfies $o(f)= [log(n)]$ and the second one satisfies $[n^{\frac{1}{c+1}}] \leq o(f) \leq [n^{\frac{1}{c}}]$. 

Finally, in \cite{HaLR18},  Hauseux and Le Roux study the polynomial entropy of Brouwer homeomorphisms. We would like to point out, that since all the points in a Brouwer homeomorphism are wandering, there is no recurrence involved in the entropy of such maps. In their work, they define the wandering  polynomial entropy of a map. They prove that a Brouwer homeomorphism has wandering polynomial entropy $1$ if and only if it is conjugate to a translation. No Brouwer homeomorphism has wandering polynomial entropy in the open interval $(1,2)$. And, for every $\alpha\in[2,\infty]$  there exists a Brouwer homeomorphism $f_\alpha$ with wandering polynomial entropy $\alpha$. 

Their results can also be translated and extended to our context. However, since this is a work in progress of the first author with a PhD student, we are going to update this article with the appropriate reference later on. 

Having  now discussed previous works, we question if only studying polynomial orders of growth is sufficient to understand maps with vanishing entropy. Since we have  complete picture of homeomorphisms of the circle, we move into studying generalized entropy on surfaces.


In our following example, we are going to construct a family of transitive maps, all of them with $0$ topological entropy and such that the generalized entropies form an interesting set in $\OOR$.  We also would like to argue that studying the generalized entropy is necessary, and that polynomial entropy is not enough. 

Our next theorem talks about the generalized entropy of cylindrical cascades. For us, a cylindrical cascade is a map $f:S^1\times \R \to S^1\times \R $ of the form $f(x,y)=(x+\alpha, y +\varphi(x))$ where $\varphi:S^1\to \R$ is a $C^1$ map. We will call $\Cy$ the family of cylindrical cascades. When studying cylindrical cascades is commonly considered higher dimension and higher regularity. However, for our purposes this set up is going to be sufficient. A relevant fact about cylindrical cascades that we would like to point out that, is that all of them are isotopic to the identity.

Dynamical properties of these maps have been studied by many researchers. Recurrence in higher dimension has been studied by Yoccoz in \cite{Yo80}, \cite{Yo95} and Chevalier and Conze in \cite{ChCo09}. Transitivity has been studied by Gottschalk and Hedlund in \cite{GoHe55} and examples were given by Sidorov in \cite{Si73}. Ergodic properties have been studied by Krygin in \cite{Kr74} and Conze in \cite{Co09} for the case $S^1\times \R$. For higher dimension Conze in \cite{Co80} worked in the case of fibers in the Heisenberg group and most notably Cirilo and Fayad have announced genericity of ergodic maps in the general case $\T^d\times \R^r$.  

Since $S^1\times \R$ is not a compact space, we would like to observe that this is not a problem. By the definition of the cylindrical cascades, we could very well project them in $\T^2$ and work there. Or, we could also define generalized topological entropy in non-compact spaces in the same way is done by Bowen in \cite{Bo71}. Since the projection from $S^1\times \R$ to $\T^2$ is a local isometry both solutions are equivalent. This is, both a cylindrical cascade and its projection, have the same generalized entropy. We would like to clarify that Theorem \ref{Inv} also holds in the non-compact case, yet only for uniformly continuous conjugations. For more details on the non-compact case, check subsection \ref{NCCSS}. 

In the following theorem, we construct cylindrical cascades with arbitrarily slow generalized entropy.   

\begin{teoprin}\label{CyCa}
For every $o\in \OOR$ there exists a cylindrical cascade $f\in \Cy$ such that $f$ is transitive and $0<o(f)\leq o$. Moreover, the maps in $\Cy$ which verify this, are dense in $\Cy$.  
\end{teoprin}

This theorem implies that for the family of cylindrical cascades, polynomial entropy is not sufficient. If we consider $o= \inf(\pol)$, then we obtain a dense set of maps in $\Cy$ with $0$ polynomial entropy. 

We would like to compare our approach with another natural perspective in measuring the separation of orbits in a family of dynamical systems. Given a order of growth $[b(n)]$ we can construct the one parameter family of orders of growth $\B=\{[b(n)^t]:0<t<\infty\}$. The set $\B$ is a natural generalization  of the sets $\E$ and $\pol$. In fact, if $b(n) = e^n$, then $\B = \E$ and if $b(n)= n$, then $\B =  \pol$.  If we define 
\[h_\B(f)= \lim_{\epsilon\to 0}\limsup_n \frac{log(g_{f,\epsilon}(n))}{log(b(n))},\] 
then by the arguments of Theorem \ref{Rel} we deduce that $\pi_{\B}(o(f))= h_\B(f)$.

This gives a natural approach: Given a family of dynamical systems $\Ha$, instead of working with $o(f)$, find a order of growth $[b(n)]$ such that for any $f$ in $\Ha$, $0 < h_\B(f) < \infty$. This perspective is tempting because dealing with  $\lim_{\epsilon\to 0}\limsup_n \frac{log(g_{f,\epsilon}(n))}{log(b(n))}$ seems technically easier than $o(f)$.  We have two objections to this. First, from our experience, computing $o(f)$ it is not much more difficult than computing $h_\B(f)$ for maps with $0$ topological entropy. Also, by Theorem \ref{CyCa} this approach is not enough for the family of cylindrical cascades. Given a order of growth $[b(n)]$ we know $0 < [log(b(n))]<\inf(\B)$. By Theorem \ref{CyCa}, there exists a dense set of maps in $\Cy$ such that $0 < o(f) \leq [log(b(n))]$. This implies that for any $\B$, there exists a dense set in $\Cy$ with $h_\B(f) =0$. Because of this, we conclude that in order to understand how cylindrical cascades separate orbits, we need to study their generalized topological entropy. Figure \ref{FigB} represents the previous argument. 

\begin{figure}[!htp]

\centering
\includegraphics{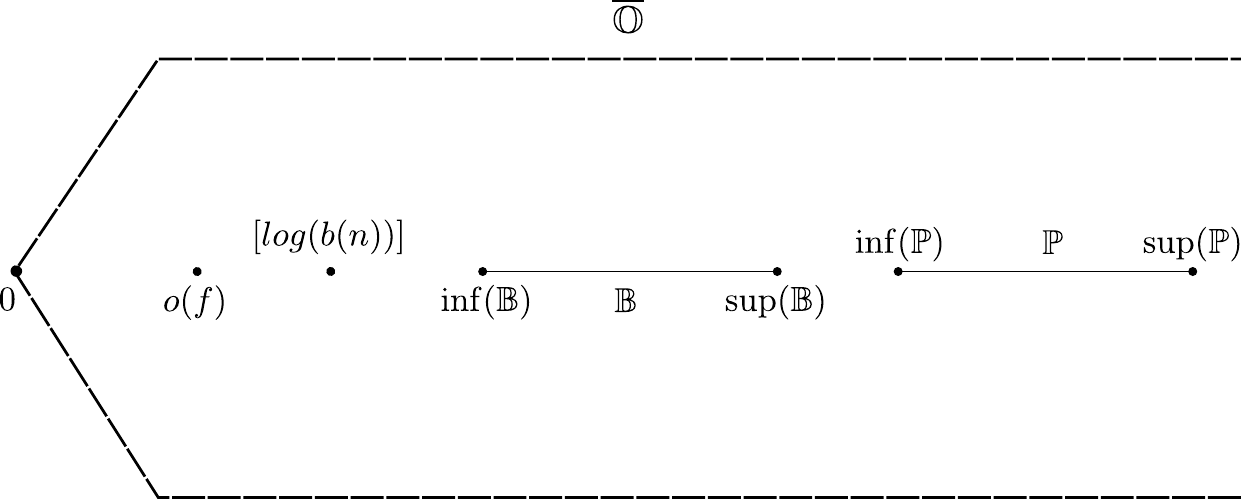}
\caption{generalized entropy of cylindrical cascades and 1-parameter families of orders of growth.}
\label{FigB}
\end{figure}


We would like now to show, how the concept of generalized entropy allows us to formulate new questions, and also enrich our perspective. Let us recall Shub's entropy conjecture and what is known so far. Given $M^m$ a manifold of dimension $m$ and $f:M\to M$ a diffeomorphism, for each $k$ in $\{0,\dots, m\}$, consider $f_{*,k}: H_{k}(M,\R)\to  H_{k}(M,\R)$ the action induced by $f$ on the real homology groups of $M$. If $sp(f_{*,k})$ is the spectral radius of $f_{*,k}$ and $sp(f_*)=\max\{sp(f_{*,k}): 0\leq k \leq dim(M)\}$, then Shub conjectured in \cite{Sh74} that
\[log(sp(f_*))\leq h(f).\]

Manning proved in \cite{Ma75} that the weaker inequality $log(sp(f_{*,1}))\leq h(f)$ always holds for homeomorphisms in any dimension. In particular, this implies that the conjecture is always true for homeomorphisms for $m\leq 3$. This result was then improved by Bowen in  \cite{Bo78} where he studied the action in the first fundamental group instead of the first homology group. 

 From the work of Palis, Pugh, Shub and Sullivan in \cite{PPSS} and Kirby and Siebenmann in \cite{KiSi69} can be concluded that the conjecture holds for an open and dense subset of the space of homeomorphisms when $m\neq 4$.

Marzantowicz, Misiurewicz and Przytycki proved in \cite{MaPr08}, \cite{MiPr77a} that the conjecture is also known to hold for homeomorphisms on any infra-nilmanifold. Some weaker versions of the conjecture were proved by Ivanov in \cite{Iv82}, Misiurewicz and Pryztycki in \cite{MiPr77b} and Oliveira and Viana in \cite{OlVi08}. 

A major progress in the conjecture was obtained by Yomdin in \cite{Yo87} where he proved the conjecture for every $C^\infty$ diffeomorphism. When restricted to classes of dynamical systems with some kind of hyperbolicity, the conjecture was proved by Shub and Williams in \cite{ShWi75}, Ruelle and Sullivan in \cite{RuSu75}, and Saghin, Xia in \cite{SaXi10}. So far, the strongest statement of this kind is the one from Liao, Viana and Yang in \cite{LiViYa13} where they prove the conjecture for every diffeomorphisms away from tangencies.  

What is lacking in this context, is a description for maps such that $sp(f_*)=1$. We would like to observe that the environment of generalized topological entropy provide us a language to study such problem. Our following theorem is a contribution to this topic.

\begin{teoprin}\label{Shub}
Let $M$ be a manifold of finite dimension and $f:M\to M$ be a homeomorphism. If $sp(f_{*,1})=1$, then there exists $k$ which only depends on $f_{*,1}$ such that $[n^k] \leq o(f)$. Moreover, $k$ is computed as follow. Consider $J$ the Jordan normal form associated to a matrix that represents $f_{*,1}$. Let $k_\R$ be the maximum dimension among the Jordan blocks associated to either $1$ or $-1$. Let $k_\C$ be the maximum dimension among Jordan blocks associated to complex eigenvalues. Then, $k = \max\{k_\R, k_\C /2\} - 1$. 
\end{teoprin}

We would like to recall that the examples built in Theorem \ref{CyCa} can be projected into $\T^2$ and all of them are isotopic to the identity. In particular, for the identity, $k=0$ and no lower bound can be obtained in this category. We compile this information in the following corollary. 

\begin{coroprin}
In the space $Hom(\T^2)$ there are three categories:
\begin{itemize}
\item Either $f_{*,1}$ is hyperbolic and $log(sp(f_{*,1}))\leq h(f)$.
\item Or  $f_{*,1}$ is a Dehn twist and $[n]\leq o(f)$. 
\item Or $f_{*,1}$ is a matrix of the form
\[A = 
\begin{pmatrix}
\pm 1 & 0 \\
0 & \pm 1
\end{pmatrix},
\] 
in which there are elements in the isotopy class with arbitrarily slow generalized entropy. 
\end{itemize}
\end{coroprin}


We are going to conclude this introduction with a few observations that were left behind, an example and some questions we consider interesting. 

Let us compute the generalized entropy of an example which is somehow related to the previous theorem. We are going to consider skew-products in the annulus $S^1\times[0,1]$ where the map is the identity in the base $[0,1]$ and rotations of different angles on the fibers $S^1$. Since the identity map in the interval and the rotations of circles are all Lyapunov stable, on each piece the map has $0$ generalized entropy. Therefore, it could be expected that the skew-product also has $0$ generalized entropy. However, this is not the case. 

\begin{ex}\label{ExSPA}
Consider the annulus $\A=S^1\times [0,1]$, $\alpha:[0,1]\to [0,1]$ a continuous increasing map and $R_{\alpha(t)}:S^1\to S^1$ the rotation in the circle of angle $\alpha(t)$. If $f:\A\to \A$ is the homeomorphism defined by $f(s,t)=(R_{\alpha(t)}(s),t)$, then $f$ has entropy numbers $([n],[n])$.\end{ex}

Observe that this example and Denjoy maps have both the same generalized entropy. Yet, the dispersion of orbits of both come from very different structures. The separation of orbits in Denjoy maps come from an expansive dynamic in a cantor set, where the dispersion in the skew-product comes form invariant dynamics moving at different speeds. This shows that generalized entropy is a sensitive tool, and that understanding the phenomena which causes positive generalized entropy could be a delicate problem.

Topological entropy can also be defined using separated sets and this arises a natural question. If we define generalized topological entropy with separated sets, do both de\-fi\-ni\-tions coincide? The answer is yes and we prove this in subsection \ref{NCCSS}. The same question can be asked for open coverings and the answer is also yes. We separate the proof of this to Appendix A, because throughout this article, we are not going to use open coverings.
 
It came to our attention that Egashira in \cite{Eg93} made a similar construction to ours in the context of foliations which was later translated by Walczak in \cite{Wa13} to the context of group actions. He does indeed define orders of growth classes and then ``completes'' his space. However, his orders of growth classes are different from our since he allows comparison between some sub-sequences. On the other hand, he completes his space of orders of growth by considering the abstract limit of sequences of ordered classes. This way of completing the space, if translated to our construction, might result in possibly a smaller set. 

An important topic we have not discussed yet is metric entropy. A first difficult in this topic is the choice of a definition for generalized metric entropy. The classical approach through partitions is inconvenient, mainly because Kolgomorov-Sinai theorem  can not be translated. This implies that in order to compute the generalized metric entropy of a map, one has to understand the metric entropy in every partition. Another interesting fact, is that, if a variational principle happens to be true, then it will only hold in the closure of the union of the supports of the invariant measure and not in the whole space. This observation can be seen in example \ref{ExBE}.

There are many interesting families of dynamical systems with vanishing entropy to study in the context of generalized entropy. We wonder what can be said about reparametrizations of irrational flows, cherry flows, unimodal maps or the quadratic family to mention a few. A question we propose in this topic is the realization of orders of growth. This is, given $\Ha$ a family of dynamical systems such that  $o_i = \inf\{o(f):f\in \Ha\}$ and $o_s = \sup\{o(f):f\in \Ha\}$, does for every $o_i \leq o \leq o_s$ exists $f\in \Ha$ such that $o(f)=o$ ? 

It is also intriguing for us to know how dynamical properties interact with $o(f)$. Theorem \ref{Ent0} and proposition \ref{PropExpHomeo} are results in this vein. We expect that topological mixing or weak-mixing have some impact on $o(f)$. Reciprocally, we would like to know if there is a setting such that positive generalized entropy implies certain growth in the number of periodic orbits. The cylindrical cascades with an irrational rotation have no periodic points, yet we believe that none of them has generalized entropy beyond $[n]$. The examples from \cite{HaLR18} have  only one fixed point and generalized entropy between $[n^2]$ and $\sup(\pol)$. And therefore, some type of recurrence like transitivity is probably required. 

Another important topic in entropy is continuity, and for this we have little hope. One of the problems is that since $\OOR$ is very big, the order topology in $\OOR$ is bad. Also, Theorem \ref{CyCa} shows how chaotic the function $f\to o(f)$ can be. In this family, we expect at most that there is some type of upper continuity in the $C^\infty$ topology.  

This paper is structured as follows: In section \ref{GTE}, we prove Theorems \ref{Inv} and \ref{Rel}. In section \ref{VE}, we prove Theorems \ref{Ent0} and \ref{ClaHS1}, and we also construct and explain examples \ref{ExBE} and \ref{ExSPA}. In section \ref{CAS}, we prove Theorem \ref{CyCa} and in section \ref{SHUB} we prove Theorem \ref{Shub}. Then, in Appendix A we study the generalized topological entropy from the open coverings point of view. Finally, in Appendix B we review some of the classical properties of topological of entropy in the context of generalized topological entropy. 	

\section{Generalized Topological Entropy.} \label{GTE}

In this section, we are going to study the generalized topological entropy of continuous maps. First,  we are going to develop the generalized topological entropy through the point of view of $(n,\epsilon)$-separated sets and we are also going to study the non-compact case. With this, we can prove Theorem \ref{Inv} and \ref{Rel}. 

\subsection{Non compact case and separated sets.}\label{NCCSS}

We start by observing that we can also define the generalized topological entropy of a system when $M$ is not a compact set. We do this in an analogous way as in the definition of entropy. Given $M$ a metric spaces, $f:M\to M$ a continuous map and $K\subset M$ a compact set we say that $E\subset K$ is   $(n,\epsilon)$-generator of $K$ if $K \subset \bigcup_{x\in E} B(x,n,\epsilon)$. Then, we define $g_{f,K,\epsilon}(n)$ equal to the minimum cardinality of an $(n,\epsilon)$-generator of $K$, $G_{f,K} =\{[g_{f,K,\epsilon}(n)]\in \OR:\epsilon > 0\}$ and $o(f,K)= \sup(G_{f,K})\in \OOR$. Finally, we define $o(f) = \sup\{o(f,K)\in \OOR:K\subset M\ is\ compact\}$. 

Another important observation is that the notion of entropy can be defined through $(n,\epsilon)$-separated sets. We will define another generalized entropy through this perspective and see that both notions coincide. Given $M$ a metric space, $f:M\to M$ a continuous map and $K\subset M$ a compact set we say that $E\subset K$ is $(n,\epsilon)$-separated if $B(x,n,\epsilon)\cap E = \{x\}$ for all $x\in E$. We define $s_{f,K,\epsilon}(n)$ the maximal cardinality of an $(n,\epsilon)$-separated set. Analogously, as with $g_{f,K,\epsilon}$ we know that $s_{f,K,\epsilon}$ is a non-decreasing sequence of natural numbers. Then, we define $S_{f,K} = \{[s_{f,K,\epsilon}(n)]\in \OR:\epsilon > 0\}$ and $u(f,K)=\sup(S_{f,K})\in \OOR$. Finally, we define $u(f) = \sup\{u(f,K)\in \OOR:K\subset M\ is\ compact\}$. 

\begin{prop} \label{SepEquGen}
Let us consider $M$ a metric space and $f:M\to M$ a continuous map. If $K\subset M$ a compact set, then $o(f,K) = u(f,K)$. In particular, $o(f) = u(f)$.
\end{prop}

The proof of this proposition is a consequence of the following lemma:

\begin{lema}\label{LemSepEquGen}
$g_{f,K,\epsilon}(n) \leq s_{f,K,\epsilon}(n)\leq g_{f,K,\epsilon/2}(n)$ for all $n\geq 1$, for all $\epsilon >0$ and for all compact $K\subset M$. 
\end{lema}

A proof of this lemma can be found \cite{Bo71}.

\begin{proof}[Proof of \ref{SepEquGen}]
By Lemma \ref{LemSepEquGen}, we deduce that $[g_{f,K,\epsilon}(n)] \leq [s_{f,K,\epsilon}(n)]\leq [g_{f,K,\epsilon/2}(n)]$. The first inequality implies that $ u(f,K)\leq o(f,K)$, and the second one implies that $ o(f,K)\leq u(f,K)$. From this we conclude that $o(f,K)=u(f,K)$.
\end{proof}

\subsection{\texorpdfstring{$o(f)$}{o(f)} is a topological invariant - proof of Theorem \ref{Inv}.} \label{TeoInv}

\begin{proof}[ Proof of Theorem \ref{Inv}]
Consider $f:M\to M$ and $g:N\to N$, two continuous maps such there exists $h:M\to N$ a homeomorphism which satisfies $h\circ f = g \circ h$. Given $\epsilon>0$, consider $\delta>0$ from the uniform continuity of $h$. Let $E$ be an $(n,\epsilon)$-separated set of $g$ such that $s_{g,\epsilon}(n) = \#E$. We claim that $h^{-1}(E)$ is an $(n,\delta)$-separated set of $f$. If this was not true, then there exists $x_1,x_2\in h^{-1}(E)$ and $k\leq n$ such that $d(f^k(x_1),f^k(x_2))\leq \delta$. By the continuity of $h$, we know that  $d(h(f^k(x_1)),h(f^k(x_2)))\leq \epsilon$. Using that $h$ conjugates $f$ and $g$ we see that $d(g^k(h(x_1)),g^k(h(x_2)))\leq \epsilon$ which contradicts that $E$ is an $(n,\epsilon)$-separated set of $g$. 

If  $h^{-1}(E)$ is an $(n,\delta)$-separated set of $f$, we infer that $s_{f,\delta}(n) \geq \# h^{-1}(E) = \# E= s_{g,\epsilon}(n)$. In particular, $[s_{f,\delta}(n)]\geq [s_{g,\epsilon}(n)]$ and from this we deduce that $o(f)\geq o(g)$. Since $h$ is a homeomorphism we analogously prove that $o(f)\leq o(g)$ and then we conclude that $o(f)= o(g)$.  
\end{proof}

We would like to point out that this theorem also holds for the non compact case when the conjugacy is uniformly continuous.

\subsection{Relationship between \texorpdfstring{$o(f)$}{o(f)} and \texorpdfstring{$h(f)$}{h(f)} - proof of Theorem \ref{Rel}.} \label{TeoRel}

In order to prove that $\pi_{\E}(o(f))=h(f)$ we would like to do two things: first, recall the definition of $\pi_\E:\OOR\to [0,\infty]$. Once we consider the interval $I_{\E}(o)= \{t\in (0,\infty):o \leq [exp(tn)]\}\subset \R$ we define $\pi_{\E}(o)=\inf (I_\E (o))$ if $I_{\E}(o)\neq \emptyset$ and $\pi_\E(o)=\infty$ otherwise. Second, to point out the following lemma which we are not going to prove. 

\begin{lema}
The following four are equivalent:
\begin{enumerate}
\item $[a_1(n)]\leq [a_2(n)]$ (There exists a constant $c$ such that $a_1(n) \leq C a_2(n)$ for all $n$).
\item $\liminf_n \frac{a_2(n)}{a_1(n)} >0$.
\item $\limsup_n \frac{a_1(n)}{a_2(n)} <\infty$.
\item There exists a constant $c$ and $n_0$ such that $a_1(n) \leq c.a_2(n)$ for all $n\geq n_0$.
\end{enumerate}
\end{lema}

\begin{proof}[Proof of $\pi_{\E}(o(f))=h(f)$] 

Let us suppose that $h(f)<\infty$. If this happens, then 
\[\limsup_{n} \frac{1}{n} log(g_{f,\epsilon}(n)) \leq h(f)\ \forall \epsilon >0.\]
which implies  
\[\limsup_{n} \frac{g_{f,\epsilon}(n)}{exp((h(f)+\delta)n)} < \infty\ \forall \epsilon >0\ \forall \delta >0.\]
This means, that $[g_{f,\epsilon}(n)]  \leq [exp((h(f)+\delta)n)]$ and therefore $ o(f) \leq  [exp((h(f)+\delta)n)]$. In particular, $\pi_\E(o(f))\leq (h(f)+\delta)$ for any $\delta$ and then $\pi_\E(o(f))\leq h(f)$. Moreover, it is also concluded that if $h(f)<\infty$, then $\pi_{\E}(o(f))<\infty$.

Let us suppose that $\pi_{\E}(o(f))<\infty$. Recalling that this means  $I_\E (o(f)) \neq \emptyset$, we take $t_0$ such that $ o(f) \leq [exp(t_0 n)]$. Then, $  [g_{f,\epsilon}(n)]\leq  [exp(t_0 n)]$ for all $\epsilon>0$. This implies that
\[\liminf_n \frac{exp(t_0 n)}{g_{f,\epsilon}(n)} > 0,\]
which is equivalent to
\[\limsup_n \frac{g_{f,\epsilon}(n)}{exp(t_0 n)} < \infty,\]
and therefore
\[\limsup_n \frac{1}{n} log (g_{f,\epsilon}(n)) \leq t_0.\]

Since this holds for all $\epsilon >0$ we infer that $h(f) \leq t_0$ and then $h(f) \leq \inf (I_\E (o(f)))=\pi_{\E}(o(f))$. Moreover, it is also concluded that if  $\pi_{\E}(o(f))<\infty$, then $h(f)<\infty$.

From the previous two arguments we deduce two things. First, that $h(f)<\infty$ if and only if $\pi_{\E}(o(f))<\infty$. And second, if one of those happens, then $h(f)= \pi_\E(o(f))$.

\end{proof}

We would like to observe that Theorem \ref{Rel} also holds when $M$ is not compact. 

From Theorem \ref{Rel}, it remains to see that $o(f)\leq \sup(\E)$. Since we are going to use later the main argument to prove this, we would like to put it separate. This argument will allow us to compute upper bounds for $s_{f,\epsilon}(n)$.

\begin{lema}\label{LemPhiInj}
Let $M$ be a compact metric space and $f:M\to M$ a continuous map. Let us fix $\epsilon>0$ and suppose that $M$ can be covered by  $B_1, \dots, B_k$  balls of radius $\epsilon/2$. Let $E$ be an $(n,\epsilon)$-separated set and $\varphi:E\to \{1,\dots,k\}^n$ a map which associates to each point an itinerary. This means, if $\varphi(x)= (i_0,\dots,i_{n-1})$, then  $f^j(x)\in B_{i_j}$. Then, $\varphi$ is injective.
\end{lema}

\begin{proof}
If not, we would have two points $x,y$ in $E$ such that $d(f^i(x),f^i(y))<\epsilon$ for all $0\leq i\leq n-1$. This contradicts the fact that $E$ is an $(n,\epsilon)$-separated set.
\end{proof}

We will call maps like $\varphi$ itinerary maps.

\begin{proof}[Proof of $o(f)\leq \sup(\E)$]
Let us fix $\epsilon>0$ and consider $B_1, \dots, B_k$  balls of radius $\epsilon/2$ which covers $M$. Take $E$ an $(n,\epsilon)$-separated set with $\#E = s_{f,\epsilon}(n)$ and $\varphi:E\to \{1,\dots,k\}^n$ an itinerary map as in lemma \ref{LemPhiInj}. We know by this lemma that $\varphi$ is injective and therefore $s_{f,\epsilon}(n)\leq k^n$. Since $k$ depends on $\epsilon$ and $k(\epsilon)\to \infty$ as $\epsilon\to 0$ we conclude that
\[o(f)=\sup\{[s_{f,\epsilon}(n)]:\epsilon>0\} \leq \sup\{[k(\epsilon)^n]:\epsilon>0\} = \sup(\E).\]
\end{proof}

\section{Maps with vanishing entropy.}\label{VE}

In this section, we prove first Theorem \ref{Ent0}. Then, we study the generalized entropy of homeomorphisms of the circle, which means proving  Theorem \ref{ClaHS1}. Finally, we are going to construct and explain examples \ref{ExBE} and \ref{ExSPA}.

\subsection{Lyapunov stable maps - Proof of Theorem \ref{Ent0}.}\label{TeoEnt0}

\begin{proof}[Proof of Theorem \ref{Ent0}] We are going to prove first that $o(f)=0$ if and only if $f$ is Lyapunov Stable. 

$(\Longrightarrow):$
By the definition of Lyapunov stability, given $\epsilon>0$, there exists $\delta$ such that if $d(x,y)<\delta$, then $d(f^n(x),f^n(y))<\epsilon$ for all $n\in \N$. In particular, $B(x,\delta)\subset B(x,n,\epsilon)$. Since $M$ is compact there exists $x_1,\dots,x_k$ points in $M$ such that $\{B(x_i,\delta):i=1,\dots,k\}$ is a covering of $M$. By the previous, we know that $\{B(x_i,n,\epsilon):i=1,\dots,k\}$ is a covering of $M$ and therefore $g_{f,\epsilon}(n)\leq k$. This implies that $[g_{f,\epsilon}(n)] = 0$ and therefore $o(f)=0$.

$(\Longleftarrow):$
Suppose that $o(f)=0$, and observe that given $\epsilon>0$ we conclude that $[s_{f,\epsilon}(n)] \leq o(f)=0$, and therefore $[s_{f,\epsilon}(n)] =0$. 
This implies that $s_{f,\epsilon}$ is a bounded sequence.

Since $s_{f,\epsilon}$ is a bounded non-decreasing sequence it is eventually constant. Let us say $s_{f,\epsilon}(n) = s_{f,\epsilon}(n_0)$ for all $n\geq n_0$.  If a set $E$ is $(n,\epsilon)$-separated, then it is $(m,\epsilon)$-separated for all $m>n$. From this, we see that we can take $E$, an $(n,\epsilon)$-separated set such that $s_{f,\epsilon}(n)=\#E$ for all $n\geq n_0$.  

Recall that if an $(n,\epsilon)$-separated set is such that  its cardinal is $s_{f,\epsilon}(n)$, then  it also is an $(n,\epsilon)$-generator. In particular, we know that for every $n\in \N$ and $x\in M$ there exists $y_n\in E$ such that $d_n(x,y_n)<\epsilon$. Since $E$ is finite and $d_m(x,y)\geq d_n(x,y)$ if $m>n$, then we can deduce that for every $x\in M$ there exists $y\in E$ such that $d_n(x,y)< \epsilon$ for all $n\in \N$.

To simplify the notation, we will call $d_\infty(x,y)= \sup\{d_n(x,y):n\in\N\}$.

Let us prove the result. Suppose by contradiction that there exists $\eta>0$ such that for every $m$ there exists $x_m,y_m$ verifying $d(x_m,y_m)<1/m$ and $d_\infty(x_m,y_m) >\eta$. Taking a sub-sequence if necessary, we can consider that there exists $z\in M$ and $x,y\in E$ such that $x_m\to z$, $y_m\to z$, $d_\infty(x_m,x)\leq \epsilon$ and $d_\infty(y_m,y) \leq \epsilon$. Since each $d_n$ is  continuous we conclude that $d_\infty(z,x)\leq \epsilon$ and $d_\infty(z,y)\leq \epsilon$. Then, we infer that
\[d_\infty(x_m,y_m)\leq d_\infty(x_m,x) + d_\infty(x,z) + d_\infty(z,y) + d_\infty(y,y_m) \leq 4\epsilon,\]
and if we take $\epsilon < \eta/4$ we have a contradiction.

Let us show now that if there exists $x\in M$ such that $x\notin \alpha(x)$, then $o(f) \geq [n]$. If $x\notin \alpha(x)$, then there exists $\epsilon>0$ such that $d(x,f^{-n}(x))\geq \epsilon$ for all $n\geq 1$. Given $n\in \N$, we claim that $\{f^{-i}(x):0\leq i <n\}$ is an $(n,\epsilon)$-separated set. From this, $s_{f,\epsilon}(n)\geq n$ and then  $o(f)\geq [s_{f,\epsilon}(n)]\geq [n]$ which implies the result. To prove the claim, observe that given $0\leq i<j\leq n-1$, $d_n(f^{-i}(x),f^{-j}(x))\geq d(f^i(f^{-i}(x)),f^i(f^{-j}(x)))=d(x,f^{i-j}(x))>\epsilon$.
\end{proof}

\subsection{Homeomorphisms of the Circle - Proof of Theorem \ref{ClaHS1}} \label{SClaHS1}

We are going to split the proof of Theorem \ref{ClaHS1} in two lemmas. 

\begin{lema}\label{GTEHS1}
Let $f:S^1\to S^1$ be a homeomorphism. If $f$ is not Lyapunov stable, then $o(f)=[n]$. 
\end{lema}

\begin{proof}

Let us start by proving that $o(f)\leq [n]$. The argument is the same used for proving that if $f$ is a homeomorphism of the circle, then $h(f)=0$. Let us fix $\epsilon>0$ and consider a finite covering of $S^1$ by intervals of length $\epsilon$. Suppose that $I_1, \dots, I_k$ are such intervals and let $E$ be an $(n,\epsilon)$-separated set with $\#E = s_{f,\epsilon}(n)$. Again, we consider an itinerary map $\varphi:E\to \{1,\dots,k\}^n$. We know by Lemma \ref{LemPhiInj} that $\varphi$ is injective. The difference here with respect to the second part of Theorem \ref{Rel}, is that we can prove $\# \varphi(E)\leq 4k.n$. Let us consider an admissible itinerary $(i_1,\dots,i_n)$. In particular, $\bigcap_{j=0}^{n-1}f^{-j}(I_{i_j})\neq \emptyset$ and therefore it is an interval with two endpoints. Observe also, that each endpoint is an endpoint of a $f^{-j}(I_{i_j})$. Since we have $k.n$ intervals $f^{-j}(I_{i})$ with $0\leq j\leq n-1$ and $1\leq i\leq k$, we know that $\# \varphi(E)\leq 4k.n$. Therefore $ s_{f,\epsilon}(n)\leq 4kn$, which implies $[s_{f,\epsilon}(n)]\leq [n]$, and then $o(f) \leq [n]$.

We need now to prove $o(f)\geq [n]$ and for this we are going to use the second part of Theorem \ref{Ent0}. Let us consider first the case when $f$ reverses orientation. If  $f$ reverses orientation, then $f$ has two fixed points. Now we have two possibilities for the remaining points: they are all periodic of period two or there are wandering points. For the first case, $f$ is Lyapunov stable and for the second, by Theorem \ref{Ent0} we deduce that $o(f)\geq [n]$.

Let us study the case when $f$ preserves orientation. In this case, we have well defined the rotation number $\rho(f)$. If $\rho(f) =p/q\in \Q$, we know that: $f$ has periodic points; they all have period $q$; and the no wandering set of $f$ consists only of these periodic points. We now have two possibilities. If $\Omega(f)=S^1$, then $f$ is Lyapunov stable. If $\Omega(f)\neq S^1$, we have wandering points and again by the second part of Theorem  \ref{Ent0} we conclude.

 If $\rho(f) \notin \Q$, we know that $f$ is semi-conjugate to an irrational rotation. Since the rotation is Lyapunov stable, if $f$ is in fact conjugate, then $f$ is also Lyapunov stable. If not, $f$ has wandering points and analogously by the second part of Theorem \ref{Ent0} we have finished.
\end{proof}

From the previous lemma, we could not separate a Denjoy Map (when $f$ is only semi-conjugate to the irrational rotation) from a Morse-Smale (the no-wandering set only consists of a finite number of hyperbolic periodic points). For this we have the following result:

\begin{lema}\label{GTEDM}
Let $f:S^1\to S^1$ be a homeomorphism. If $f$ is a Denjoy Map, then $o(f,\Omega(f))=[n]$. 
\end{lema}

\begin{proof}
Since $o(f,\Omega(f))\leq o(f)\leq [n]$, it only remains to prove that $o(f,\Omega(f))$ $\geq [n]$. 

We will call wandering intervals to the connected components of $S^1\setminus\Omega(f)$. Let us consider $\epsilon>0$ such there exists some wandering interval of length greater than $\epsilon$. We define now $A_1=\{I_1,\dots, I_k\}$ the collection of all the wandering intervals of length greater than $\epsilon$. We know that this is finite set because $S^1$ has finite length. We proceed to define by induction the sets  $A_{n+1}= f^{-1}(A_n)\cup A_1$. Observe that for $n$ big enough,  $\#A_n\geq n$. This is true because the intervals are wandering and therefore at each step we have to add at least one new interval. To prove the result, let us observe that the two points $x,y$ in the border of an interval of $A_n$ belong to $\Omega(f)$ and $d_n(x,y)=\sup\{d(f^i(x),f^i(y)):\ 0\leq i< n\}>\epsilon$. Now, if we take for each connected component of $S^1\setminus \cup_{I\in A_n} I $ one point in the border, then by the previous argument we obtain an $(n,\epsilon)$-separated set. This set has $\#A_n$ points and therefore $s_{f,\Omega(f),\epsilon}(n)\geq \#A_n$ which implies $[n]\leq [s_{f,\Omega(f),\epsilon}(n)]\leq o(f,\Omega(f))$. 
\end{proof}

Let us now prove Theorem \ref{ClaHS1}

\begin{proof}[Proof of Theorem \ref{ClaHS1}] 

For each homeomorphism of the circle $f$ we associate the tuple $(o(f,\Omega(f)),$ $o(f))$. For the Denjoy maps, we obtain $([n],[n])$ and for Lyapunov stable maps, we obtain $(0,0)$. For the rest, since $\Omega(f)=Per(f)$ we obtain $(0,[n])$. In particular, by the criteria defined in the introduction we infer that Denjoy maps are more dispersive than Morse-Smale maps which are more dispersive than rotations. 
\end{proof}

\subsection{Example \ref{ExBE} - A map with different entropy numbers.}\label{SExBE}

This example is inspired by \cite{HaLR18} and Bowen's eye map. 

Consider $f:\D^2\to \D^2$ the time $1$ map of a flow as in figure \ref{FigFlow}.

\begin{figure}[!htp]

\centering
\includegraphics[scale=0.6]{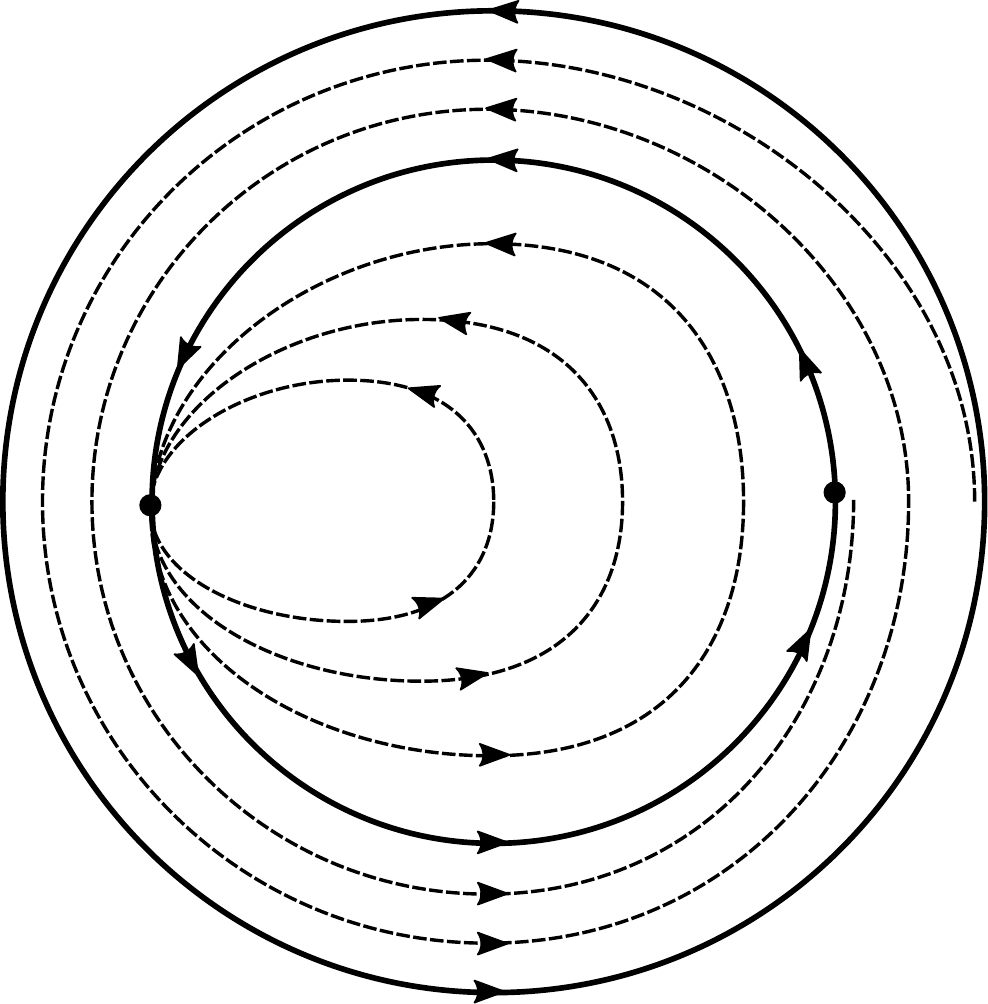}
\caption{Flow for example \ref{ExBE}}
\label{FigFlow}
\end{figure}

Observe that in this flow there are two invariant regions, the inner disk and the outer ring. We will call $D$ the inner disk and $C$ the border of said disk. The purpose of the outer ring is to make the inner circle part of the no-wandering set. In particular, in $C$ there are two singularities, which induces two parabolic fixed points $p_1$ and $p_2$.  Its not hard to see, that $\overline{Rec(f)}=\{p_1,p_2\}$ and $\Omega(f) =C \cup \partial \D^2$. From this, we conclude that $o(f, \overline{Rec(f)})= 0$ and $o(f,\Omega(f))=[n]$. 

It remains then to prove that $o(f)\geq [n^2]$. In order to do so, we are going to use the technique developed in \cite{HaLR18}. We are not applying directly their theory because the setting is different, but the main argument still holds. 

Let us consider two open sets $U_0$ and $U_1$ inside $D$ such that:
\begin{itemize}
\item $U_0$ and $U_1$ are wandering sets.
\item $U_0$ is in the lower half of the disk $D$.
\item $U_1$ is in the upper half of the disk $D$.
\item $\partial U_0 \cap C \neq \emptyset$ and  $\partial U_1 \cap C \neq \emptyset$.
\end{itemize}  

Given $\epsilon>0$ we consider $V_0=\{x\in U_0:d(x, D \setminus U_0)>\epsilon \}$ and $V_1=\{x\in U_1:d(x,D\setminus U_1)>\epsilon \}$. If $\epsilon$ is small enough, then $V_0$ and $V_1$ are not empty and moreover $\partial V_i\cap C \neq \emptyset$ for $i=0,1$. 

We would like to code the  orbits in $int(D)$. For this, we define $R= D\setminus (V_0\cup V_1)$. Now, we fix $n$ and consider the  itinerary map $\varphi_n:int(D)\to \{V_0,V_1,R\}^{n+1}$. 

We claim to have the following property:\emph{ There exists $k_0>0$ such that for all $n$, for all $k_0 \leq l\leq n$ and for all $0\leq i\leq n-l$ there exists $x\in D$ with $\varphi_n(x)=(w_0,\dots, w_n)$ verifying $w_i = V_0$, $w_{i+l}= V_1$ and $w_j= R$ for all $j\neq i, i+l$.} 

The reason for this claim to be true, is that the speed of the flow near the singularity becomes arbitrarily close to $0$ and we have points in $V_0$ and $V_1$ as near as $C$ as we want. 

For each pair  $(i,l)$ we consider the point $x_{i,l}$ as in the claim. If we define $S_n=\{x_{i,l}:k_0\leq l \leq n, 0\leq i\leq n- l \}$, then
 $S_n$ is an $(n,\epsilon)$-separated set. To see this, let us consider $x_{i,l},x_{i',l'}\in S$ with $\varphi_n(x_{i,l})=(w_0,\dots, w_n)$ and $\varphi_n(x_{i',l'})=(w'_0,\dots, w'_n)$. Suppose that $i\neq i'$ (the other case is analogous) and observe that $w_i = V_0$, $w'_i=R$ and $w'_{i'}= V_0$. Since $U_0$ is a wandering set and $w'_{i'}= V_0$, we infer that $f^i(x_{i',l'})\notin U_0$ and therefore   $x_{i,l}$ and $x_{i',l'}$ are $(n,\epsilon)$-separated. 

By a simple computation we see that $[\#S_n] = [n^2]$ and therefore $[n^2]\leq [s_{f,\epsilon}(n)]\leq o(f)$.

\subsection{Example \ref{ExSPA} - Generalized entropy of some twist maps.}\label{SExSPA}

We are going study the generalized topological entropy of some  of homeomorphisms of the annulus $\A=S^1\times [0,1]$. In particular, those twist-maps which leaves the circles $S^1\times \{t\}$ invariant for all $t\in[0,1]$. 

To simplify the computation we are going to use in $\A$ the metric \[d((s_1,t_1),(s_2,t_2))=\max\{d(s_1,s_2),|t_1-t_2|\}.\]

The first thing we are going to do, is lift $f$. Let us consider $\pi:\R\times[0,1]\to \A$ the natural projection, $\alpha:[0,1]\to [0,1]$ a continuous increasing map and $f:\A\to \A$ defined by $f(s,t)=(R_{\alpha(t)}(s),t)$. The map $F:\R\times[0,1]\to \R\times[0,1]$ such that $F(s,t)=(s+\alpha(t),t)$ is a lift of $f$ which satisfies $f\circ \pi = \pi \circ F$. Observe that if $D=[0,1]\times [0,1]\subset \R\times [0,1]$, then $o(F,D)=o(f)$. This is true because the entropy is locally computed, $\pi$ is a local isometry and $F$ and $f$ are conjugated. Moreover, due to the fact that $F$ commutes with the action of the fundamental group of $\A$, we deduce that $o(F)=o(f)$. 

The reason we are going through this is the following: Given two points $x,y\in \A$, in order to know if they are $(n,\epsilon)$-separated we need to know all the values $d(x,y),$ $d(f(x),f(y)),$ $\dots,\ d(f^n(x),f^n(y))$. Now, if $f$ is a twist map, any curve which is transverse to the horizontal direction in every point, is stretched in every iterate of $f$. This implies that given any two close points 
$x=(s_1,t_1), y=(s_2,t_2)\in \R\times[0,1]$  with $t_1\neq t_2$, if $d(F^k(x),F^k(y))>\epsilon$, then $d(F^n(x),F^n(y))>\epsilon$ for all $n>k$. 
Although, this might not happen to $f$, the $(n,\epsilon)$-balls are isometric by $\pi$ and therefore this does not contradict the claim  $o(F,D)=o(f)$. In particular, the previous implies that we only need  to consider the value of $d(F^n(x),F^n(y))$ to know if two close points,  that do not belong to the same horizontal line, are $(n,\epsilon)$-separated.

We say that $\beta:[a,b]\to \R\times[0,1]$ is a vertical curve if $[a,b]\subset [0,1]$ and $\beta(t)=(s_0,t)$ for a fixed $s_0\in \R$. 

\begin{lema}\label{LemGenIso}
Suppose that $F(s,t)=(s+\alpha(t),t)$ with $\alpha:[0,1]\to \R$ an increasing map. Given $\epsilon>0$, consider $0\leq s_1<\dots<s_l\leq 1$  such that $s_{i+1}-s_i \leq \epsilon/2$ and also  $\beta_1,\dots,\beta_l:[a,b]\to \R\times [0,1]$ the vertical curves associated to $\{s_1,\dots,s_l\}$. Then, there exists $G_\epsilon(n)$ an $(n,\epsilon)$-generator of $[0,1]\times [a,b]$ such that 
\[\#G_\epsilon(n)= \left\lceil \frac{2.l.n.(\alpha(b)-\alpha(a))}{\epsilon}\right\rceil.\]
\end{lema}

\begin{proof}
Observe that $d(F^n(s,t),F^n(\hat s, \hat t))= \max\{|n(\alpha(\hat t)-\alpha(t)) + \hat s -  s|,|\hat t -  t|\}$. If we consider $ t_1=a<t_2<\dots <t_q=b$ such that 
\begin{equation}\label{eqGen}
|n(\alpha(t_{i+1})-\alpha(t_{i}))|\leq \epsilon/2,
\end{equation}
then $G_\epsilon(n)= \{\beta_i(t_j):1\leq i \leq l, 1\leq j \leq q\}$ is an $(n,\epsilon)$-generator of $[0,1]\times[a,b]$.

 Since $\#G_\epsilon(n)=l.q$ we just need to compute $q$. For this, we add on $i$ in equation \ref{eqGen} and infer that $n(\alpha(b)-\alpha(a))\leq q .\epsilon/2$. Since $\alpha$ is continuous we can take such $t_i$ verifying $q=\left\lceil 2.n.(\alpha(b)-\alpha(a))/\epsilon \right\rceil,$ and from this, we have finished. 

\end{proof}

We also have: 

\begin{lema}\label{LemSepIso}
Suppose that $F(s,t)=(s+\alpha(t),t)$ with $\alpha:[0,1]\to \R$ an increasing map. Given $\epsilon>0$, consider $0\leq s_1<\dots<s_l\leq 1$  such that $s_{i+1}-s_i > \epsilon$ and also  $\beta_1,\dots,\beta_l:[a,b]\to \R\times [0,1]$ the vertical curves associated to $\{s_1,\dots,s_l\}$. Then, there exists $S_\epsilon(n)$ an $(n,\epsilon)$-separated set of $[0,1]\times [a,b]$ such that 
\[\#S_\epsilon(n)= \left\lfloor  \frac{l.n.(\alpha(b)-\alpha(a))}{\epsilon}\right\rfloor.\]
\end{lema}

The proof of this lemma is analogous as the proof of lemma \ref{LemGenIso} and therefore we are going to omit it. 

Since $\Omega(f)=\A$, we just need to prove that $o(f)=[n]$. Now, by our previous arguments, we just need to see $o(F,D)=[n]$, where $D=[0,1]\times[0,1]$.

Given $\epsilon>0$, consider $G_\epsilon(n)$ as in lemma \ref{LemGenIso}. We know that $g_{F,D,\epsilon}(n) \leq \#G_\epsilon(n)$ and therefore $[g_{F,D,\epsilon}(n)]\leq [n]$. This implies $o(F,D)\leq[n]$. Analogously, we consider $S_\epsilon(n)$ as in lemma \ref{LemSepIso}. Since $s_{F,D,\epsilon}(n) \geq \#S_\epsilon(n)$, we infer that $[s_{F,D,\epsilon}(n)]\geq [n]$, and then  $o(F,D)\geq[n]$.

\section{Cylindrical Cascades.} \label{CAS}


As mentioned in the introduction, a cylindrical cascade for us is a map $f:S^1\times \R \to S^1\times \R $ of the form $f(x,y)=(x+\alpha, y +\varphi(x))$ where $\varphi:S^1\to \R$ is a $C^1$ map. We are going to work with cylindrical cascades through the classical approach, that is, study $\varphi$ as the limit of trigonometric polynomials. To prove Theorem \ref{CyCa}, we are going to construct an example with $o(f)\leq [a(n)]$ for some fixed sequence $a(n)$, and then explain why $f$ is transitive and why we can build this type of examples in a dense set of $\Cy$.

Let us start by considering an irrational number $\alpha$  and $\{p_k/q_k\}_{k\in \N}$ the sequence of Diophantine approximations. Consider also a sequence $b_k$ which decreases to $0$ and define $ \varphi_k:S^1\to [-1,1]$ by $\varphi_k(x) = b_k cos(2\pi q_k x)$ and then 
\[f(x,y)= \left(R_\alpha(x), y +\sum_k \varphi_k(R_\alpha(x))\right).\]

The sequences $\{p_k/q_k\}_{k\in \N}$ and $\{b_k\}_{k\in \N}$ are not going to arbitrary. In fact, their speed of convergence is going to be our variable in order to obtain the result. We are going to discuss throughout the proof what conditions $b_k$ and $q_k$ needs to verify. At the end of the construction, we are going to explain the process of choosing this numbers such that all conditions are verified. To begin with, we need
\begin{equation}\label{CC1}
\sum_k b_k q_k< \infty, 
\end{equation}
for $f$ to be a $C^1$ map.  

Let us represent the Weyl sum of $\varphi_k$ under the rotation $R_\alpha$ by $S_n(\varphi_k) = \sum_{j=0}^{n-1} \varphi_k\circ R^j_\alpha$. With this, when $f$ is iterated we see that
\[f^n(x,y)= \left(R^n_\alpha(x), y +\sum_k S_n(\varphi_k)(x)\right).\]

\subsection{Upper-bound for \texorpdfstring{$o(f)$}{o(f)}.}

Our goal now is to construct $f$ such that $o(f)\leq [a(n)]$. For this to happen, our strategy is going to be set by the following lemma. 

\begin{lema}
Suppose $a(n)$ is a non decreasing sequence such that $\sum_{k} |S_n(\varphi'_k)|\leq a(n)$. Then, $o(f)\leq [a(n)]$.
\end{lema}

\begin{proof}
The map $f$ does not separate orbits in the vertical axis, so we need to compute the separation of orbits in the horizontal axis. Let us fix $\epsilon$ and consider two points $x_1,x_2 \in S^1$ such that $d(x_1,x_2)\leq \epsilon/a(n)$. By a simple computation, we deduce that 
\[ d(f^j(x_1,y),f^j(x_2,y))\leq \sum_{k} |S_j(\varphi'_k)|d(x_1,x_2)\leq \epsilon.a(j)/a(n)\leq \epsilon.\]

Therefore, $(x_1,y)$ and $(x_2,y)$ are not $(n,\epsilon)$-separated. This implies that $s_{f,\epsilon}(n)\leq a(n)/\epsilon^2$, and then $[s_{f,\epsilon}(n)]\leq [a(n)]$ $\forall \epsilon$. From this we conclude that $o(f)\leq [a(n)]$. 
 
\end{proof}

This lemma, give us a way to bound $o(f)$  which is to compute $|S_n(\varphi'_k)|$. 

\subsection{Known facts about Diophantine approximations.}        

Let us briefly recall some classical properties of Diophantine approximations. If 
\[\alpha = \frac{1}{r_1 + \frac{1}{r_2 + \frac{1}{r_3+ \frac{1}{\dots}}}},\]
then $q_{k+1} = r_{k+1} q_k + q_{k-1}$. Since $\frac{q_{k-1}}{q_k} \leq 1$,  we infer the estimate
\begin{equation}\label{rk}
r_{k+1} \approx \frac{q_{k+1}}{q_k}.
\end{equation}
 
We also know that
\[\frac{1}{(r_{k+1}+2)(q_k)^2}\leq \left|\alpha  - \frac{p_k}{q_k}\right|\leq  \frac{1}{r_{k+1}(q_k)^2},\] 
which implies 
\begin{equation}\label{est1}
\frac{1}{(r_{k+1}+2)q_k}\leq \left\|q_k \alpha\right\|\leq  \frac{1}{r_{k+1}q_k},
\end{equation} 
where $\left\|q_k \alpha\right\|$ is the distance in the circle between the projection of $0$ and $q_k \alpha$.

\subsection{Upper-bounds for the Weyl sum of the derivatives.}       

In this subsection, we are going to show two things. First, we obtain a constant upper-bound for $|S_n(\varphi_k')|$ for any $n$ and second we obtain a linear upper-bound for up to certain integer. 

The upper-bound we get for $|S_n(\varphi_k')|$ comes from the fact that $\varphi_k'$ has a solution for the co-homological equation and therefore the orbit of a point moves along the graph of said solution.

\begin{lema} $|S_n(\varphi_k')|\leq 4\pi b_k q_k q_{k+1}$, for every $n\in \N$.
\end{lema}

\begin{proof}
To prove this, we start by observing that we can write $\varphi_k$ as 
\[ \varphi_k(x)= \frac{b_k}{2}\left(exp(2\pi q_k i x) + exp(-2\pi q_k i x)\right).\] 
If we define the map 
\[ u_k(x)=\frac{b_k}{2}\left(\frac{exp(2\pi q_k i x)}{exp(2\pi q_k i  \alpha)-1} + \frac{exp(-2\pi q_k i x)}{exp(-2\pi  q_k i \alpha)-1}\right),\]
then we deduce that $\varphi_k(x) = u_k(x+\alpha)-u_k(x)$. Therefore, $S_n(\varphi_k')(x)=u_k'(x+n\alpha)-u_k'(x)$ and since
\[u_k'(x)=\frac{b_k 2\pi q_k i}{2}\left(\frac{exp(2\pi q_k i x)}{exp(2\pi q_k i \alpha)-1} - \frac{exp(-2\pi q_k i x)}{exp(-2\pi  q_k i \alpha)-1}\right),\]
we infer that
\[|S_n(\varphi_k')|\leq 2|u_k'|\leq \frac{4\pi b_k q_k}{|exp(2\pi q_k i \alpha)-1|}.\]

Now $|exp(2\pi q_k i \alpha)-1|$ is in fact $\left\|q_k \alpha\right\|$ and by equations \ref{rk} and \ref{est1} we see that
\[|S_n(\varphi_k')|\leq 4\pi b_k q_k^2 (r_{k+1} + 2) \approx 4\pi b_k q_k q_{k+1}.\]
\end{proof}

Since we are studying orders of growth, the constant $4\pi$ can be ignored. Therefore, from now on we are going to assume:
\begin{equation}\label{uppbnd}
|S_n(\varphi_k')|\leq b_k q_k q_{k+1}\quad \forall n\in \N. 
\end{equation}


We proceed to show that $|S_n(\varphi_k')|$ has a linear upper-bound for up to certain integer. 

\begin{lema}\label{lnrgrwth} If $n\leq \frac{\sqrt{q_{k+1}}}{\pi \sqrt{2 b_k q_k}}$, then $|S_n(\varphi_k')|\leq 2\pi b_k q_k n +1$. 
\end{lema}

\begin{proof}
To prove this, given $x\in S^1$ we compare $S_n(\varphi_k')(x)$ with $- 2 \pi b_k q_k sen(2\pi q_k x)n$. Recall that $S_n(\varphi_k')(x) = \sum_{j=0}^{n-1} \varphi'_k\circ R^j_\alpha (x) = \sum_{j=0}^{n-1} - 2\pi b_k q_k sen(2\pi q_k ( x + j\alpha))$ and therefore
\[\left|S_n(\varphi_k')(x) - \left(- 2 \pi b_k q_k sen(2\pi q_k x)n\right) \right|\leq 2\pi b_k q_k \sum_{j=0}^{n-1} \left| sen(2\pi q_k ( x + j\alpha)) - sen(2\pi q_k x)\right|. \]
Now, by Mean Value theorem we deduce that 
\[\left|S_n(\varphi_k')(x) - \left(- 2 \pi b_k q_k sen(2\pi q_k x)n\right) \right|\leq 2\pi b_k q_k \sum_{j=0}^{n-1} 2\pi q_k j \left\| q_k \alpha \right\| \leq 2\pi^2 b_k q_k^2 n^2 \left\| q_k \alpha \right\|,\]
and if we combine this with equations \ref{rk} and \ref{est1}, we conclude that
\[ 2\pi^2 b_k q_k^2 n^2 \left\| q_k \alpha \right\| \leq \frac{2\pi^2 b_k q_k n^2}{r_{k+1}} \approx  \frac{2\pi^2 b_k q_k n^2}{q_{k+1}}. \]
By the previous, as long as $n$ is such that $\frac{2\pi^2 b_k q_k n^2}{q_{k+1}} \leq 1$, we know that
\[ |S_n(\varphi_k')(x)|\leq |2 \pi b_k q_k sen(2\pi q_k x)n|+ |S_n(\varphi_k')(x) -(- 2 \pi b_k q_k sen(2\pi q_k x)n)|  \leq  2 \pi b_k q_k n+ 1.\]
This is, $|S_n(\varphi_k')|\leq 2 \pi b_k q_k n + 1$ up to $n\approx \frac{\sqrt{q_{k+1}}}{\pi \sqrt{2 b_k q_k}}$.
\end{proof}

Again, we are going to ignore the constants that do not depends on $k$ and $n$. And so we are going to work with the equations:
\begin{equation}\label{lnr1}
|S_n(\varphi_k')|\leq b_k q_k n,
\end{equation}
up to 
\begin{equation}\label{lnr2}
n_k = \frac{\sqrt{q_{k+1}}}{ \sqrt{b_k q_k}}. 
\end{equation}

Since we are going to want that $\lim_k n_k = \infty$ we need 
\[\lim_k \frac{\sqrt{q_{k+1}}}{ \sqrt{b_k q_k}} = \infty.
\] 
This is given by the fact that $\lim_k q_k= \infty $ and  $\lim_k b_k q_k =0$ (equation \ref{CC1}).

We resume the information obtained in the previous two lemmas within Figure \ref{FigUpprbnd}.

\begin{figure}[!htp]
\centering
\includegraphics{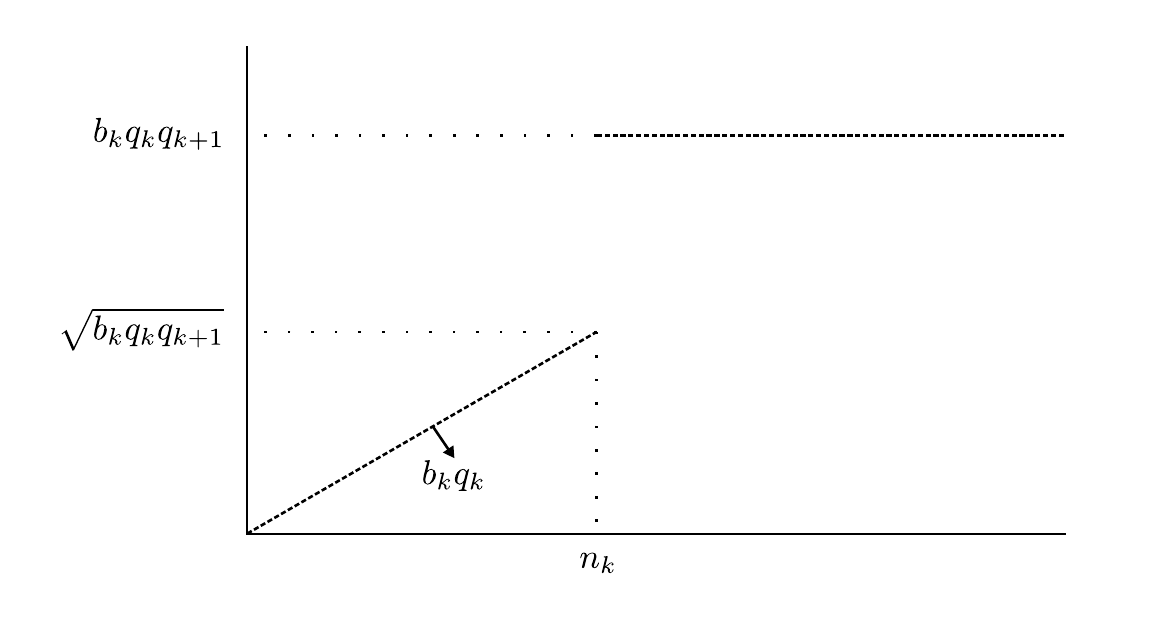}
\caption{Upperbound for $|S_n(\varphi_k')|$.}
\label{FigUpprbnd}

\end{figure}

\subsection{Sum of the upper-bounds.}        

We proceed now to add all of these upper bounds on $k$. Although, it seems natural to add up this bounds on each interval $[n_{k-1}, n_k]$, since $n_k$ depends on $b_k$, working on such intervals would be troublesome for the inductive construction. Because of this, let us take a sequence $m_k$ such that $m_k\leq n_k$ and define the intervals $I_1 = [0,m_1]$ and $I_k= [m_{k-1}, m_k]$. We are going to cut the linear bound on $m_k$ and therefore, on each $I_k$, the upper bounds add up to a function $C_k n + D_k$, where 
\[ C_k = \sum_{j\geq k} b_j q_j,\]
 and 
\[D_k = \sum^{k-1}_{j=1} b_j q_j q_{j+1}.\]
Figure \ref{FigSumupprbnd} illustrates this piecewise linear sequence.  

\begin{figure}[!htp]

\centering
\includegraphics{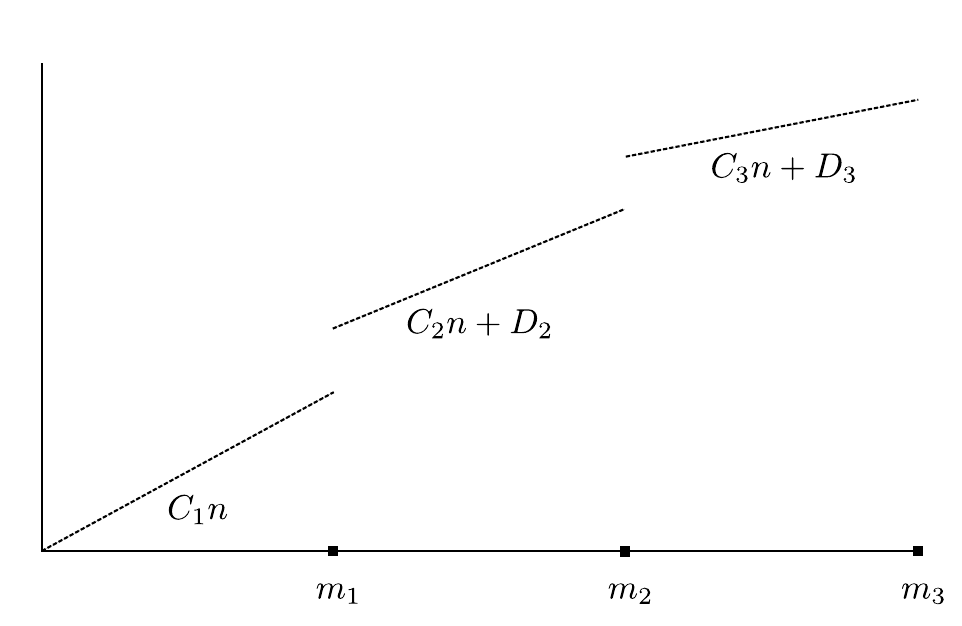}
\caption{Sum of the upper bounds}
\label{FigSumupprbnd}
\end{figure}
Observe that $C_k$ is the tail of the convergent series $\sum_k b_k q_k$, and therefore the slopes in this piecewise linear sequences tend to $0$. 

Let us call $e(n) = C_k n + D_k$ if $m_{k-1}\leq n < m_k$. Our goal now is to choose $b_k$ and $q_k$ such that $[e(n)]\leq [a(n)]$. The construction is going to be by induction on $k$. However, since $n_k$ depends on $q_{k+1}$, $q_{k+1}$ has to be chosen in step $k$.

\subsection{Inductive construction for the upper-bound.}        

For each $k$, and for each $j \leq k$, define $C_j^k = \sum^{i=k}_{i= j} b_i q_i$ and then for every $n\leq m_k$ we define $e^k(n) = C_j^k n + D_j$ if $m_{j-1} \leq n < m_j$. We need to do this, because $C_j$ depends on future $b_k$ and $q_k$. We also define $e^k(m_k)= D_k$. Once this is set, our inductive hypothesis is 
\[e^k(n) < a(n) - \frac{1}{2^k} \quad \forall n \leq m_k.\]

 Since $e(n) = \lim_k e^k(n)$, if this holds, then we infer that $e(n) \leq a(n)$ and therefore that $[e(n)] \leq [a(n)]$. 

Suppose that $b_k$, $q_{k+1}$ and $m_k$ have been chosen such that  $e^k(n) < a(n) - \frac{1}{2^k}\  \forall n \leq m_k$. Fix some big $m_{k+1}$ and if $q_{k+2}$ is such that $\frac{\sqrt{q_{k+2}}}{\sqrt{q_{k+1}}}> m_{k+1}$, since $b_{k+1}$ is going to be smaller than $1$, 
\begin{equation} \label{CUB1}
n_{k+1}> m_{k+1}.
\end{equation}

We are first going to see which restrictions we need for $b_{k+1}$ such that the inductive hypothesis in the step $k+1$ holds up to $m_k$. We know that $C_j^k n + D_j < a(n) - \frac{1}{2^k}$ if $n \leq m_k$, and we want 
\begin{equation}\label{CUB2}
C_j^{k+1} n + D_j < a(n) - \frac{1}{2^{k+1}}\text{ if }n \leq m_k.
\end{equation}
 Now $C_j^{k+1}= C_j^k + b_{k+1}$, and thus we maintain our inductive property up to $m_k$ if we ask for $b_{k+1} m_k <1/2^{k+1}$.

Now, for $n$ between $m_k$ and $m_{k+1}$ we want 
\begin{equation} \label{CUB3}
e^{k+1}(n) = C_{k+1}^{k+1} n + D_{k+1} < a(n) - \frac{1}{2^{k+1}} \text{ if } m_k\leq n < m_{k+1}.
\end{equation}
 and 
\begin{equation} \label{CUB4}
e^{k+1}(m_{k+1})= D_{k+2} =\sum^{k+1}_{j=1} b_j q_j q_{j+1} \leq a(m_{k+1}) - \frac{1}{2^{k+1}}.
\end{equation}
 We observe that the inequality in equation \ref{CUB3} depends on $b_{k+1}$ yet it does not depend on $q_{k+2}$. On the other hand, equation \ref{CUB4} depends on both $b_{k+1}$ and $q_{k+2}$. Because of this, we choose first $b_{k+1}$ and then $q_{k+2}$ for both equations to hold.

Figure \ref{FigFinalarg} illustrates the previous argument. 
\begin{figure}[!htp]

\centering
\includegraphics{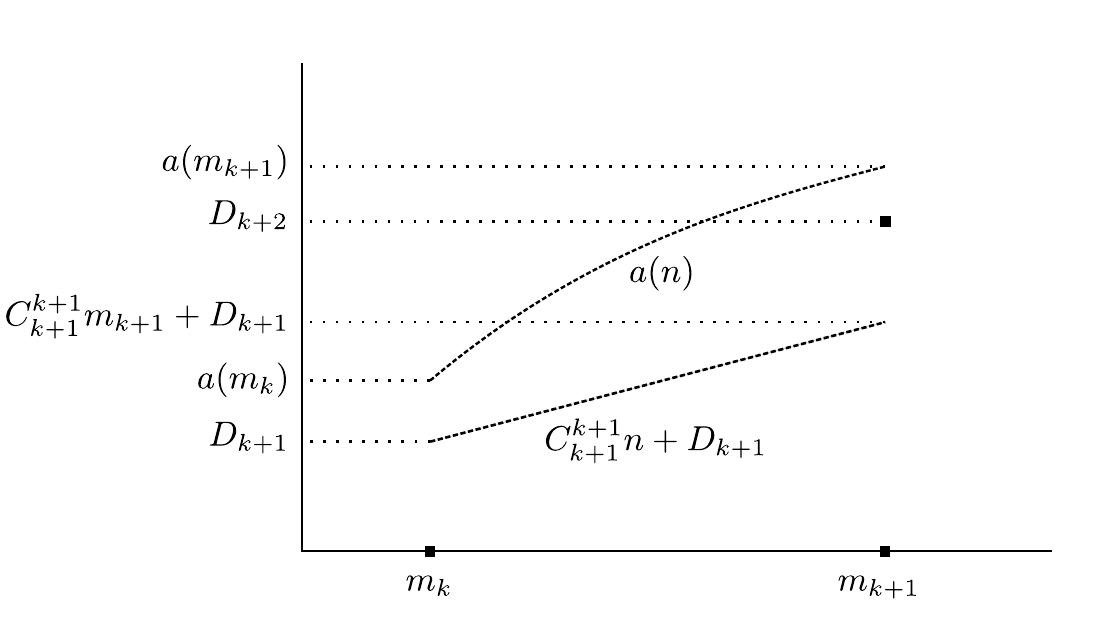}
\caption{Inductive choice of $b_{k+1}$ and $q_{k+2}$}
\label{FigFinalarg}
\end{figure}

With this, we finish our study for the conditions needed for $o(f) \leq [a(n)]$. 

\subsection{\texorpdfstring{$f$}{f} is not Lyapunov stable.}        

In this subsection, we are going to investigate which are the needed conditions for $f$ to not be Lyapunov stable. For this to happen, we need to show that $s_{f,\epsilon}(n)$ is not a bounded sequence, or equivalently that $\lim_n s_{f,\epsilon}(n) = \infty$. Again, to estimate $s_{f,\epsilon}(n)$ we are going to study $|S_n(\varphi_k')(x)|$. However, in this situation, we are not going to control the Weyl sums for every $x$ but for big subsets of $S^1$.  

The idea to prove that $o(f)>0$ follows from the arguments of lemma \ref{lnrgrwth}. We would like to observe that from the proof of said lemma we can conclude that
\[|S_n(\varphi_k')(x)| \geq |2 \pi b_k q_k sen(2\pi q_k x)n| - 1\quad \forall n\leq n_k.\]
When $n= n_k$, we infer that
\[|S_{n_k}(\varphi_k')(x)| \geq |\sqrt{2b_k q_k q_{k+1}} sen(2\pi q_k x)| - 1.\]

Now, if we consider the set $\Lambda_k = \{x\in S^1:|sen(2\pi q_k x)| > \frac{1}{\sqrt{2}} \}$ and $x\in \Lambda_k$, then 
\[|S_{n_k}(\varphi_k')(x)| \geq \sqrt{b_k q_k q_{k+1}} - 1.\]

We assert that when $x\in \Lambda_k$,  $\sum_j S_{n_k}(\varphi_j')(x)$ is comparable to $|S_{n_k}(\varphi_k')(x)|$. This happens for two reasons. For $j>k$, $|\sum_{j>k} S_{n_k}(\varphi_j')(x)|$ is going to be small. If we define $f_k(x,y)= \left(R_\alpha(x), y +\sum^k_{j=1} \varphi_j(R_\alpha(x))\right)$, then  we can choose $b_j$ for $j> k$ small enough such that 
\begin{equation}\label{CLB1}
dist_{C^1} (f_k^{n_k}, f^{n_k})\leq \frac{1}{2^k}. 
\end{equation}

For $j<k$,  
$|\sum^{j=k-1}_{j=1} S_{n_k}(\varphi_j')(x)|\leq D_k$ and when choosing $b_k$ and $q_{k+1}$ we can do it such that 
\begin{equation}\label{CLB2}
D_k \leq \sqrt{b_k q_k q_{k+1}}/2.
\end{equation}

For this choice, if $x\in \Lambda_k$, then we deduce that
\[\left|\sum_{j} S_{n_k}(\varphi_j')(x)\right| \geq \frac{1}{2} \sqrt{b_k q_k q_{k+1}} - 1.\]

Now, $\Lambda_k$ is the union of $q_k$ intervals of length $\frac{A}{q_k}$ where $A$ is a constant independent of $k$. If $I$ is one of these intervals, there are in $I$ at least $ \frac{\left(\sqrt{b_k q_k q_{k+1}} - 1\right)A}{2\epsilon q_k}$ points which are $(n,\epsilon)$-separated. If $q_{k+1}$ is big enough such that 
\begin{equation}\label{CLB3}
 \frac{A}{2\epsilon} \frac{\sqrt{b_k q_{k+1}} }{ \sqrt{q_k}} > 1,
\end{equation} 
then $s_{f,\epsilon}(n_k) \geq q_k$. Therefore, $\lim_k s_{f,\epsilon}(n_k)= \infty$ which implies that $o(f)>0$.

\subsection{Coherence in the inductive construction and final remarks.}        

It remains to verify that there is no conflict in the choices of $b_k$ and $q_k$. Let us state here the construction process. Suppose that we have chosen $m_k$, $b_k$ and $q_{k+1}$. We pick first $m_{k+1}$ big enough such that $D_{k+1} < \sqrt{a(m_{k+1}) -D_{k+1}}/2$. This restriction is for equation \ref{CLB2}. Then, we obtain $\hat q_{k+2}$ a  lower-bound for $q_{k+2}$ such that $n_{k+1}$ is going to be bigger than $m_{k+1}$ (equation \ref{CUB1}). We follow by choosing $b_{k+1}$ such that  equation \ref{CC1}, \ref{CUB2}, \ref{CUB3} and \ref{CLB1} holds. Observe that none of the previous equations depends on $q_{k+2}$, they only depend on $b_{k+1}$ and equation \ref{CUB3} depends on $m_{k+1}$ which has already been fixed. We also choose $b_{k+1}$ small enough such that we can pick $q_{k+2}> \hat q_{k+2}$ and 
\[a(m_{k+1}) - 1<   D_{k+2} = D_{k+1} + b_{k+1} q_{k+1} q_{k+2} < a(m_{k+1}) - 1/2^{k+1}.\]
The first inequality in the previous equation implies equation \ref{CLB2}. The second one implies equation \ref{CUB4}. Once we choose $q_{k+2}$ accordingly to the previous and such that equation \ref{CLB3} is verified, we have finished. 

With this, we conclude how to build an example such that $0<o(f)<[a(n)]$. It is known since \cite{GoHe55} that a cylindrical cascade is transitive if and only if the co-homological equation has no continuous solution. If this example had a continuous solution, then the orbits would move along the translated graph of said solution and then $f$ would be Lyapunov stable. This would imply that $o(f)=0$ which is a contradiction, and therefore our example is transitive.  

In order to construct a dense family of examples like these. We approximate any cylindrical cascade $f(x,y)=(R_\alpha(x), y + \varphi(x))$ by a map of the form $\hat f(x,y) =(R_\alpha(x), y + \hat\varphi(x))$ where $\hat\varphi$ is a trigonometric polynomial. We then consider $g(x,y) = (x+\alpha', y + \hat \varphi(x) + \sum_{k\geq k_0}\varphi_k(x))$ where $\alpha'$ is close to $\alpha$ and $\varphi_k$ are as the ones we have already built. The map $g$ is going to verify $0<o(f)<[a(n)]$. To see this, observe that $\hat \varphi$  has solution to the co-homological equation and therefore what it adds to the separation of orbits is finite. In particular, we can ignore it. Now, by the same argument we conclude that it is the tail of the series $\sum_{k}\varphi_k(x)$ what creates the positive and bounded generalized entropy, and therefore $g$ satisfies the desired property.  

As a final observation in this topic, we would like to point out that we could have made this construction taking sub-sequences of the $q_k$, instead of constructing them one by one. This approach would certainly light the constrictions of $\alpha$, however it would have overcharged the notation.

\section{Relationship between \texorpdfstring{$o(f)$}{o(f)} and \texorpdfstring{$f_{*,1}$}{f*1} - proof of Theorem \ref{Shub}.}\label{SHUB}

In this section, we are going prove Theorem \ref{Shub}. Let $M$ be a manifold of finite dimension and $f:M\to M$ be a homeomorphism. By the arguments of Manning in \cite{Ma75} we have the following lemma.

\begin{lema}
If $A$ is the matrix that represents the action of $f_{*,1}$, then 
\[\left\|A^{n-1}\right\|\leq 12(1 + g_{f,\epsilon}(n)).\]
\end{lema}

\begin{proof}[Proof of theorem \ref{Shub}]

By the previous lemma, we know that $[\left\|A^{n-1}\right\|] \leq o(f)$ and therefore, we must study $[\left\|A^{n-1}\right\|]$ when $sp(A)=1$. 

If $J$ is the Jordan normal form of $A$, then there exists an invertible matrix $Q$ such that $A= Q^{-1} J Q$. Since $A^n= Q^{-1} J^n Q$, we see that $[\left\|A^{n-1}\right\|] = [\left\|J^{n-1}\right\|]$. If $J_l$ are the Jordan blocks associated to $J$, then 
 $[\left\|J^{n-1}\right\|]= \sup\{[\left\|J^{n-1}_l\right\|]\}$. 

If $sp(A)=1$ and $J_l$ is associated to a real eigenvalue, then it must be either $1$ or $-1$. In any case, $J_l^n$ is a superior triangular matrix such that in the entrance $i,j$ has a number with order of growth $[n^{j - i}]$. Since the maximum value for $j-i$ is $dim(J_l) -1$, we infer that $[\left\|J^n_l\right\|]$= $[n^{dim(J_l) - 1}]$. When $J_l$ is associated to a complex eigenvalue the argument is analogous and with this we conclude the proof of theorem \ref{Shub}.

\end{proof} 
\appendix

\section{Topological entropy through open coverings.}

It is usually defined topological entropy using open coverings. Here we show how to define the generalized topological entropy this way. Let us quickly recall this approach. Given $\alpha$ an open covering of a compact space $M$, we define $H(\alpha)=log(N(\alpha))$ where 
\[N(\alpha)= \min\{\#\gamma: \gamma \subset \alpha\text{ is also covering of }M\}.\] 
If $f$ is a continuous map, then we consider 
\[\alpha^n=\{U_1\cap \dots \cap U_n: U_i\in f^{-i+1}(\alpha)\},\]
and $h(f,\alpha)=\lim_n \frac{1}{n} H(\alpha^n)$. Finally, we define 
\[h(f)=\sup\{h(f,\alpha):\alpha \text{ is an open covering}\}.\]

We translate this to our setting through the following definitions. Given $f:M\to M$ a continuous map and $\alpha$ an open covering of $M$, we define 
\[a_{f,\alpha}(n)= N(\alpha^n) = \min\{\#\gamma: \gamma \subset \alpha^n \text{ is also covering of }M\}, \] 
and then $\hat o(f)= \sup\{[a_{f,\alpha}(n)]:\alpha\text{ is an open covering of }M\}$. Before we prove $\hat o (f) = o(f)$, let us observe that
\[h(f)= \sup_\alpha  \lim_n \frac{1}{n} H(\alpha^n) = \sup_\alpha \lim_n \frac{1}{n} log (N(\alpha^n))  = \sup_\alpha \lim_n \frac{1}{n} log (a_{f,\alpha}(n)).\]
Because of the arguments of Theorem \ref{Rel}, we conclude that $\pi_\E(\hat o(f))=h(f)$.

This proves that $\hat o(f)$ is also a generalization of topological entropy. Yet, it does not prove that it coincides with our previous definition of generalized entropy. For this, we use standard arguments to show that all the definitions of topological entropy coincide. 

\begin{lema}\label{lemaAlpha}
Let $f:M\to M$ be a continuous map of a compact metric space. Given $\epsilon>0$ if $\alpha$ is an open covering of $M$ such that $diam(\alpha)<\epsilon$, then $s_{f,\epsilon}(n)\leq a_{f,\alpha}(n)$. 
\end{lema}

\begin{lema}
Let $f:M\to M$ be a continuous map of a compact metric space. Given $\alpha$ an open covering of $M$, if $\epsilon$ is a Lebesgue number of $\alpha$  then $a_{f,\alpha}(n)\leq g_{f,\epsilon}(n)$. 
\end{lema} 

A proof of both lemmas can be found in Chapter $10$ section $1.2$ of \cite{ViOl16}. The first Lemma implies that $o(f)\leq \hat o(f)$ and the second one implies that $ \hat o(f) \leq o(f)$. Therefore:

\begin{prop}
Let $f:M\to M$ be a continuous map of a compact metric space. Then, $ \hat o(f)= o(f)$.
\end{prop}

Once we know that $\hat o(f) = o(f)$, we apply Lemma \ref{lemaAlpha} again to obtain:

\begin{prop}
Let $f:M\to M$ be a continuous map of a compact metric space and $\alpha_k$ a sequence of finite open coverings such that $\lim_k diam(\alpha_k) = 0$. Then, $o(f) =``\lim_k" [a_{f,\alpha_k}(n)]= \sup\{[a_{f,\alpha_k}(n)]\}$.
\end{prop}

\section{Classical properties of topological of entropy revisited.}

In this appendix, we will see and prove some properties verified by the generalized topological entropy. None of these ones are used for the main results of the article and therefore we leave them here for a curious reader. 

The topological entropy of a map $f$ is related to the topological entropy of $f^k$ when $k\geq 1$ by the formula $h(f^k)=k.h(f)$. Since in $\OOR$ there is no additive structure such property is lost. However, at least we have:

\begin{prop}
Let $M$ be a compact metric space and  $f:M\to M$ a continuous map. The following inequalities hold:
\[o(f) \leq  o(f^2)\leq  \dots \leq o(f^k)\leq \dots.\]
\end{prop}

\begin{proof}
To prove this, observe that $g_{f^k,\epsilon}(n) = g_{f,\epsilon}(n.k)$ and since $g_{f,\epsilon}$ is non-decreasing we infer that $g_{f^k,\epsilon}(n) \geq g_{f^{k-1},\epsilon}(n)$ for all $k\geq 2$ and for all $n\geq 1$. This implies that $[g_{f^k,\epsilon}(n)]\geq  [g_{f^{k-1},\epsilon}(n)]$ and therefore $o(f^k)\geq o(f^{k-1})$.
\end{proof}

When $f$ is a homeomorphism, we know that $h(f)=h(f^{-1})$ and this property is also true for $o(f)$.

\begin{prop}
Let $M$ be a compact metric space and $f:M\to M$ a homeomorphism, then $o(f)=o(f^{-1})$.
\end{prop}

\begin{proof}
Observe that if $E$ is an $(n,\epsilon)$-separated set for $f$ then $f^{n-1}(E)$ is an $(n,\epsilon)$-se\-pa\-ra\-ted set for $f^{-1}$. From this we deduce that $s_{f,\epsilon}(n)=s_{f^{-1},\epsilon}(n)$ and then  $o(f)=o(f^{-1})$.
\end{proof}

Another interesting property of entropy is the following: Given $K_1,\dots,K_l$ a finite number of compacts sets we know that $h(f,\cup_{i=1}^l K_i) = \max\{h(f,K_i):1\leq i \leq l\}$. This is translated as following:

\begin{prop}
Let $M$ be a metric space and $f:M\to M$ continuous. If $K_1,\dots,K_l$ are a finite number of compacts sets, then $o(f,\cup_{i=1}^l K_i) = \sup\{o(f,K_i):1\leq i \leq l\}$. 
\end{prop}

\begin{proof}
Let us consider $K = \cup_{i=1}^l K_i$. Given a sequence $b(n)$ such that $o(f,K_i)\leq [b(n)]$ for all $i=1,\dots, l$, there exists $C_1,\dots, C_l$ positive constants such that
\[s_{f,K,\epsilon}(n) \leq s_{f,K_1,\epsilon}(n)+ \dots + s_{f,K_l,\epsilon}(n) \leq C_1 b(n) +\dots +C_l b(n) = (C_1+\dots +C_l) b(n).\]
From this we conclude that $o(f,K)\leq \sup\{o(f,K_i):1\leq i \leq l\}$. On the other hand, since $K_i\subset K$ we infer that $o(f,K_i)\leq o(f,K)$ for all $i$ and therefore $\sup\{o(f,K_i):1\leq i \leq l\}\leq o(f,K)$.
\end{proof}

We also know for expansive homeomorphisms that $h(f)=g(f,\epsilon)$ for some $\epsilon$ smaller that the expansivity constant. For the generalized topological entropy this result is also true.

\begin{prop}\label{PropExpHomeo}
Let $M$ be a compact metric space and $f:M\to M$ a homeomorphism. If $f$ is expansive, there exists $\epsilon$ such that $o(f) = [g_{f,\epsilon}(n)]= [s_{f,\epsilon}(n)]$. In particular, $o(f)\in \OR$.
\end{prop}
To prove this we are going to revisit the arguments of \cite{Ma87}.
We will start by pointing out the following two lemmas whose proofs can be found in Chapter $IV$ section $7$ of \cite{Ma87}.

\begin{lema}\label{LemaM1}
Let $M$ be a compact metric space and $f:M\to M$ an expansive homeomorphism with  $\epsilon_0$  an expansivity constant of $f$. If $\delta<\epsilon<\epsilon_0$, then there exists $k>0$ and $n_0>2k$ such that if $x,y\in M$ verifies
\[d(f^i(x),f^i(y))<\epsilon\quad \forall 0\leq i\leq n,\]
for some $n\geq n_0$, then 
\[d(f^i(x),f^i(y))<\delta\quad \forall k\leq i\leq n-k,\]
\end{lema}

\begin{lema}\label{LemaM2}
Let $M$ be a compact metric space and $f:M\to M$ a continuous map. If $n_1,\dots,n_j$ are positive integers and $\epsilon>0$, then 
\[g_{f,\epsilon}(n_1+\dots+n_j)\leq \prod_{i=1}^{j} g_{f,\epsilon/2}(n_i).\]
\end{lema}

\begin{proof}[Proof of \ref{PropExpHomeo}]
Let us take $\epsilon<\epsilon_0$ (the expansivity constant) and $\epsilon'<\epsilon/4$. We now apply Lemma \ref{LemaM1} to $f$ with $\delta =\epsilon'$ and we obtain $k$ and $n_0$. If $n\geq n_0-2k$, consider $E$ an $(n,\epsilon')$-separated set with $\#E=s_{f,\epsilon'}(n)$. By Lemma \ref{LemaM1}, we know that $f^{-k}(E)$ is an $(n+2k,\epsilon)$-separated set. This implies  $s_{f,\epsilon'}(n)\leq s_{f,\epsilon}(n+2k)$ and by Lemma \ref{LemSepEquGen} and Lemma \ref{LemaM1} we deduce that
\[g_{f,\epsilon'}(n)\leq s_{f,\epsilon'}(n)\leq s_{f,\epsilon}(n+2k)\leq g_{f,\epsilon/2}(n+2k) \leq g_{f,\epsilon/4}(2k) .g_{f,\epsilon/4}(n) . \]
In particular, $[g_{f,\epsilon'}(n)]\leq [g_{f,\epsilon/4}(n)]$ and by taking the supremum over $\epsilon'$ we infer that $o(f)\leq [g_{f,\epsilon/4}(n)]$. Since clearly $[g_{f,\epsilon/4}(n)]\leq o(f)$, we conclude that $o(f)= [g_{f,\epsilon/4}(n)]$.
\end{proof}


\begin{thebibliography}{99}


\bibitem{AdKoAn65}
R. Adler, A. Konheim and M. McAndrew, \emph{Topological Entropy}. Transactions of the American Mathematical Society, Volume 114, Issue 2, (1965), 309-319.


\bibitem{ArCOMo18}
A. Artigue, D. Carrasco-Olivera and I. Monteverde, \emph{Polynomial Entropy and Expansivity}. Acta Mathematica Hungarica, Volume 152, Issue 1, (2017), 152-140. 


\bibitem{BeLa16}
P. Bernard and C. Labrousse. \emph{An entropic characterization of the  flat metrics on the two to\-rus}. Geo\-me\-triae Dedicata,  Volume 180, Issue 1, (2016), 187-201.


\bibitem{BlHoMa00}
F. Blanchard, B. Host and A. Mass, \emph{Topological complexity}. Ergodic Theory and Dynamical Systems, Volume 20, Issue 3, (2000), 641-662.


\bibitem{Bo71}
R. Bowen, \emph{Entropy for Group Endomorphisms and Homogeneous Spaces}. Transactions of the American Mathematical Society, Volume 153, (1971), 401-414.


\bibitem{Bo78}
R. Bowen, \emph{Entropy and the fundamental group}. The Structure of Attractors in Dynamical Systems. Lecture Notes in Mathematics. Springer, Berlin, Heidelberg,  Volume 668, (1978), 21-29.


\bibitem{ChCo09}
N. Chevallier and J-P. Conze, \emph{Examples of recurrent or transient stationary walks in $\R^d$ over a rotation of $\T^2$}. Contemporary Mathematics, American Mathematical Society, Volume 485, (2009), 71-84. 


\bibitem{Co80}
J-P. Conze, \emph{Ergodicit\'e d'une transformation cylindrique}, Bulletin de la Soci\'et\'e Math\'e\-ma\-ti\-que de France, Volume 108, (1980), 441-456.


\bibitem{Co09}
J-P. Conze, \emph{Recurrence, ergodicity and invariant measures for cocycles over a rotation}. Contemporary Mathematics, American Mathematical Society, Volume 485, (2009), 45-70.


\bibitem{Di71}
E. Dinaburg, \emph{A correlation between topological entropy and metric entropy}. Doklady Akademii Nauk SSSR, Volume 190, Issue 1, (1971), 19-22. 


\bibitem{Eg93}
S. Egashira, \emph{Expansion growth of foliations}. Annales de la Facult\'e des sciences de Toulouse : Math\'ematiques, Serie 6, Volume 2, Issue 1, (1993), 15-52.


\bibitem{GoHe55}
W. Gottschalk and G. Hedlund, \emph{Topological Dynamics}. American Mathematical Society Colloquium Publications, Volume 36, (1955).


\bibitem{HaLR18}
 L. Hauseux and  F. Le Roux, \emph{Polynomial entropy of Brouwer homeomorphisms}. arXiv: 1712.01502.


\bibitem{Iv82}
N. V. Ivanov, \emph{Entropy and Nielsen numbers}. Doklady Akademii Nauk SSSR, Volume 265, Issue 2, (1982), 284-287. 


\bibitem{KiSi69}
R. Kirby and L. Siebenmann, \emph{On the triangulation of manifolds and the hauptvermutung}. Bulletin of the American Mathematical Society, Volume 75, Issue 4, (1969), 742-749.


\bibitem{Kr74}
A. Krygin, \emph{Examples of ergodic cylindrical cascades}. Matematicheskie Zametki,  Volume 16, Issue 6, (1974),  981-991.


\bibitem{La13}
C. Labrousse, \emph{Polynomial entropy for the circle homeomorphisms and for $C^1$ nonvanishing vector fields on $\T^2$}. arXiv: 1311.0213.


\bibitem{LiViYa13}
G. Liao, M. Viana and J. Yang, \emph{The entropy conjecture for diffeomorphisms away from tangencies}. Journal of the European Mathematical Society, Volume 15, Issue 6, (2013), 2043-2060.


\bibitem{Ma75}
A. Manning, \emph{Topological entropy and the first homology group}. Dynamical Systems - Warwick 1974. Lecture Notes in Mathematics. Springer-Verlag, Volume 468, (1975), 185-190.


\bibitem{Ma87}
R. Ma\~n\'e, \emph{Ergodic Theory and Differentiable Dynamics}. Springer-Verlag. (1987).


\bibitem{Ma13}
J-P. Marco, \emph{Polynomial entropies and integrable Hamiltonian systems}. Regular and Chaotic Dynamics, Volume 18, Issue 6, (2013), 623-655.


\bibitem{MaPr08}
W. Marzantowicz and F. Przytycki, \emph{Estimates of the topological entropy from below for continuous self-maps on some compact manifolds}. Discrete \& Continuous Dynamical Systems - Series A, Volume 21, Issue 2, (2008), 501-512.


\bibitem{MiPr77a}
M. Misiurewicz and F. Przytycki, \emph{Entropy conjecture for tori}. Bulletin de l'Acad\'emie Polonaise des Sciences. S\'erie des Sciences Math\'ematiques, Astronomiques et Physiques, Volume 25, Issue 6, (1977), 575-578.


\bibitem{MiPr77b}
M. Misiurewicz and F. Przytycki, \emph{Topological entropy and degree of smooth mappings}. Bulletin de l'Acad\'emie Polonaise des Sciences. S\'erie des Sciences Math\'ematiques, Astronomiques et Physiques, Volume 25, Issue 6, (1977), 573-574.


\bibitem{OlVi08}
K. Oliveira and M. Viana, \emph{Thermodynamical formalism for robust classes of potentials and non-uniformly hyperbolic maps}. Ergodic Theory and Dynamical Systems, Volume 28, Issue 2, (2008), 501-533.


\bibitem{PPSS}
J. Palis, C. Pugh, M. Shub and D. Sullivan, \emph{Genericity theorems in topological dynamics}. Dynamical Systems - Warwick 1974 Lecture Notes in Mathematics. Springer-Verlag, Volume 468, (1975), 241-250.


\bibitem{RuSu75}
D. Ruelle and D. Sullivan, \emph{Currents, flows and diffeomorphisms}. Topology, Volume 14, Issue 4, (1975), 319-327.


\bibitem{SaXi10}
R. Saghin and  Z. Xia \emph{The entropy conjecture for partially hyperbolic diffeomorphisms with 1-D center}. Topology and its Applications, Volume 157, Issue 1, (2010), 29-34.


\bibitem{Si73}
E. A. Sidorov, \emph{Topological Transitivity of Cylindrical Cascades}.  Matematicheskie Zametki,  Volume 14, Issue 3, (1973),  441-452.


\bibitem{Sh74}
M. Shub, \emph{Dynamical Systems, filtrations and entropy}. Bulletin of the American Mathematical Society, Volume 80, Issue 1, (1974), 27-41.


\bibitem{ShWi75}
M. Shub and R. Williams, \emph{Entropy and Stability}. Topology, Volume 14, Issue 4, (1975), 329-338.


\bibitem{ViOl16}
M. Viana and K. Oliveira,\emph{Fundations of Ergodic Theory}. Cambridge University Press. (2016).


\bibitem{Wa13}
P. G. Walczak. \emph{Expansion growth, entropy and invariant measures of distal groups and pseudogroups of homeo- and diffeomorphisms}. Discrete \& Continuous Dynamical Systems, Volume 33, Issue 10, (2013), 4731-4742.


\bibitem{Wa82}
P. Walters, \emph{An introduction to ergodic theory}. Graduate Texts in Mathematics, 79.
Springer-Verlag, New York-Berlin. (1982).


\bibitem{Yo80}
J-C. Yoccoz, \emph{Sur la disparition de propri\'et\'es de type Denjoy-Koksma en dimension 2}. Comptes Rendus des S\'eances de l'Acad\'emie des sciences, Series A-B, Volume 291, Issue 13, (1980), 655-658. 


\bibitem{Yo95}
J-C. Yoccoz, \emph{Petits diviseurs en dimension 1}. Ast\'erisque, Volume 231, (1995).


\bibitem{Yo87}
Y. Yomdin, \emph{Volume growth and entropy}. Israel Journal of Mathematics, Volume 57, Issue 3, (1987), 285-300.


\end{thebibliography}
\end{document}